\documentclass[12pt,reqno]{amsart}
\usepackage[T1]{fontenc}
\usepackage{amsfonts}
\usepackage{amssymb}
\usepackage{mathrsfs}
\usepackage[latin1]{inputenc}
\usepackage{amsmath}
\usepackage[english]{babel}
\usepackage{setspace}
\usepackage{latexsym,amsfonts,amsmath,amssymb,soul}

 \newtheorem{thm}{Theorem}[section]
 
 \newtheorem{lem}[thm]{Lemma}
 \newtheorem{prop}[thm]{Proposition}
 \theoremstyle{definition}
 \newtheorem{defn}[thm]{Definition}
 \newtheorem{rem}[thm]{Remark}

 \numberwithin{equation}{section}

 \numberwithin{equation}{section}

\newcommand{\R}{{\mathbb R}}

\newcommand{\Z}{{\mathbb Z}}
\newcommand{\C}{{\mathbb C}}
\newcommand{\N}{{\mathbb N}}

\newcommand{\cD}{{\mathcal D}}
\newcommand{\cL}{{\mathcal L}}
\newcommand{\cF}{{\mathcal F}}
\newcommand{\cE}{{\mathcal E}}

\newcommand{\cB}{{\mathcal B}}

\newcommand{\cM}{{\mathcal M}}

\newcommand{\cS}{{\mathcal S}}
\newcommand{\cO}{{\mathcal O}}

\newcommand{\su}{\subseteq}

\newcommand{\ov}{\overline}

\newcommand\proj{\mathop{\rm proj\,}}
\newcommand\ind{\mathop{\rm ind\,}}

\begin{document}

%
%
%
%

\title[Multipliers on $\cS_{\omega}(\R^N)$ ]
 {Multipliers on $\cS_{\omega}(\R^N)$}

\author[A.A. Albanese, C. Mele]{Angela\,A. Albanese and Claudio Mele}

\address{ Angela A. Albanese\\
Dipartimento di Matematica e Fisica ``E. De Giorgi''\\
Universit\`a del Salento- C.P.193\\
I-73100 Lecce, Italy}
\email{angela.albanese@unisalento.it}

\address{Claudio Mele\\
	Dipartimento di Matematica e Fisica ``E. De Giorgi''\\
	Universit\`a del Salento- C.P.193\\
	I-73100 Lecce, Italy}
\email{claudio.mele1@unisalento.it}





\begin{abstract}
The aim of this paper is to introduce and to study the space $\cO_{M,\omega}(\R^N)$ of the multipliers of the space $\cS_\omega(\R^N)$ of the $\omega$-ultradifferentiable rapidly decreasing functions of Beurling type.
 We determine various  properties of the space  $\cO_{M,\omega}(\R^N)$. Moreover,  we define  and compare some  lc-topologies of which $\cO_{M,\omega}(\R^N)$ can be naturally endowed.
\end{abstract}

\maketitle
\section{Introduction }\label{intro}

Classes of ultradifferentiable functions have been  intensively investigated
since the 20ies of the last century.
 The subject indeed has a long tradition that goes back to the work of Gevrey \cite{G}. Along the lines of \cite{G},  Komatsu introduced in  \cite{K} a way to define the ultradifferentiable functions. It consists in measuring their growth behaviour in terms of a weight sequence
$\{M_p\}_{p\in\N_0}$, which satisfies certain
conditions. Later Beurling \cite{Be} 
(see also Bj\"orck \cite{B}) showed that 
one can also use weight functions $\omega$ to measure
the smoothness of $C^\infty$-functions with compact support by the decay properties of
their Fourier transform. 
This approach was modified by Braun, Meise, and Taylor \cite{BMM},
who showed that these classes can be defined by the decay behaviour of their
derivatives by using the Young conjugate of the function $t\mapsto\omega (e^t)$.  But, in general, the classes defined in one way cannot
be defined in the other way (see, f.i., \cite{BMM}).

The study of the space of multipliers and convolutors, $\cO_M$ and $\cO'_C$, of the space $\cS(\R^N)$ of rapidly decreasing functions was
started  by Schwartz \cite{S}.  Since then the spaces $\cO_M$ and $\cO'_C$  attracted the attention of  several authors, even recently (see \cite{Ki,Ku,KM,La,LW,OW} and the references therein). On the other hand, the space of multipliers and convolutors was introduced and studied in the setting of other classes of $C^\infty$-function spaces, like  ultradifferentiable function spaces in the sense of Komatsu \cite{DPV,Ka,Kov,Pi,Pi2,Z,Z2} (see also the references therein).

In the last years  the attention  has focused on the study of the space $\cS_\omega(\R^N)$ of the ultradifferentiable rapidly decreasing functions of Beurling type, as introduced by Bj\"orck \cite{B} (see,  \cite{BJO,BJOR,BJOS}, for instance, and the references therein).  Inspired by this line of research and by the previous work, 
in this paper we introduce and study the space $\cO_{M,\omega}(\R^N)$ of the slowly increasing functions of Beurling type in the setting of ultradifferentiable function space as introduced in \cite{BMT}. In particular, we show that $\cO_{M,\omega}(\R^N)$ is the space of the  multipliers of the space $\mathcal{S}_\omega(\R^N)$ and of its dual $\cS'_\omega(\R^N)$.  We also define and compare some locally convex topologies (briefly, lc) of which   $\cO_{M,\omega}(\R^N)$ can be naturally endowed.

The paper is organized as follows. Section 2 collects some known definitions and properties about the weight functions in the sense of Braun, Meise and Taylor \cite{BMT} and the space $\cS_{\omega}(\R^N)$. Section 3 is devoted to the study of the space $\cO_{M,\omega}(\R^N)$ and of related topological properties. In Section 3 we also introduce  the space $\cO_{C,\omega}(\R^N)$ of the very slowly increasing functions of Beurling type.  In particular, we show the link between these spaces and their topological properties. In a forthcoming  paper we  prove, f.i., that   $\cO'_{C,\omega}(\R^N)$ is the space of convolutors of $\cS'_\omega(\R^N)$ and the Fourier exchange between the spaces $\cO'_{C,\omega}(\R^N)$ and $\cO_{M,\omega}(\R^N)$.
In Section 4 we show that $\cO_{M,\omega}(\R^N)$ is the space of the multipliers of $\mathcal{S}_{\omega}(\R^N)$. Finally, in Section 5 we study and compare some lc-topologies of which $\cO_{M,\omega}(\R^N)$ can be  endowed in a natural way.

\section{Weight functions and the space $\mathcal{S}_{\omega}(\mathbb{R}^N)$}

We begin with the definition of non-quasianalytic weight function in the sense of Braun-Meise-Taylor \cite{BMT} suitable
for the Beurling case, i.e., we also consider the logarithm as a weight function.

\begin{defn}\label{D.weight} A non-quasianalytic weight function is a continuous increasing function $\omega:[0,\infty)\to[0,\infty)$  satisfying the following properties:
	\begin{itemize}
		\item[($\alpha$)] there exists $K\geq1$ such that $\omega(2t)\leq K(1+\omega(t))$ for every $t\geq0$;
		\item[($\beta$)] $\int_{1}^{\infty} \frac{\omega(t)}{1+t^2}\, dt < \infty$;
		\item[($\gamma$)] there exist $a \in \mathbb{R}$, $b>0$ such that $\omega(t) \geq a+b\log(1+t)$, for every $t\geq0$;
		\item[($\delta$)] $\varphi_\omega(t)= \omega \circ \exp(t)$ is a convex function.
	\end{itemize}
\end{defn}

Given a non-quasianalytic weight function $\omega$, we can extend it on $\C^N$ by setting
$\omega(z)=\omega(|z|)$  for all $z \in \C^N$, where $\vert \cdot \vert$ denotes the standard euclidean norm.

\begin{rem}\label{wheight prop gamma'}\rm 
We now recall some known facts on weight functions that shall be useful in the following; the proofs can be found in the literature. 

	Let $\omega$ be a non-quasianalytic weight function. Then the following properties are satisfied.

(1) Condition $(\alpha)$ implies  for every $t_1, t_2 \geq 0$ that
	\begin{equation}\label{sub}
	\omega(t_1+t_2)\leq K(1+\omega(t_1)+\omega(t_2)).
	\end{equation}
	Observe that this condition is weaker than subadditivity (i.e.,  $\omega(t_1+t_2)\leq \omega(t_1) + \omega(t_2))$. The weight functions satisfying ($\alpha$) are not necessarily subadditive in general.
	
	 (2) Condition $(\alpha)$ implies that there exists $L\geq 1$ such that for every $t\geq 0$
	\begin{equation}\label{l}
	\omega(e t)\leq L(1+\omega(t)).
	\end{equation}
	
 (3) By condition $(\beta)$ and the fact that $\omega$ is an increasing function, we have that $\omega(t) = o(t)$ as $t\to\infty$.
	This can be deduced by the fact that for every $t>0$ we have
	\begin{equation*}
	\frac{\omega(t)}{t}=\int_{t}^{\infty} \frac{\omega(t)}{s^2}\, ds\leq \int_{t}^{\infty} \frac{\omega(s)}{s^2}\, ds\,.
	\end{equation*}
	\end{rem}

Let $\omega$ be a non-quasianalytic weight function. We define the Young conjugate $\varphi^*_\omega$ of $\varphi_\omega$ as the function $\varphi^*_\omega:[0,\infty)\to [0,\infty)$ defined by
\begin{equation}\label{Yconj}
\varphi^*_\omega(s):=\sup_{t\geq 0}\{ st-\varphi_\omega(t)\},\quad s\geq 0.
\end{equation}
There is no loss of generality to assume that $\omega$ vanishes on $[0,1]$. So,  $\varphi^*_\omega$ has only non-negative values. By Fenchel-Moreau Theorem (see, f.i., example \cite{BL}), we have that $\varphi^*_\omega$ is convex and increasing,   $\varphi^*_\omega(0)=0$ and $(\varphi^*_\omega)^*=\varphi_\omega$. Further useful properties of $\varphi^*_\omega$  are listed in the following lemma, see \cite{BMT}.

\begin{lem}\label{L.Pfistar} Let $\omega\colon [0,\infty)\to [0,\infty)$ be a non-quasianalytic weight function. Then the following properties are satisfied.
	\begin{enumerate}
		\item[\rm (1)] $\lim_{t\to\infty}\frac{t}{\varphi^*_\omega(t)}=0$.
		\item[\rm (2)] $\frac{\varphi^*_\omega(t)}{t}$ is an increasing function in $(0,\infty)$.
		\item[\rm (3)] For every $s,t\geq 0$ and $\lambda>0$
		\begin{equation}
		2\lambda\varphi^*_\omega\left(\frac{s+t}{2\lambda}\right)\leq \lambda \varphi^*_\omega\left(\frac {s}{\lambda}\right)+\lambda\varphi^*_\omega\left(\frac{t}{\lambda}\right) \leq \lambda\varphi^*_\omega\left(\frac{s+t}{\lambda}\right).
		\end{equation}
		\item[\rm (4)] For every $t\geq 0$ and $\lambda>0$
		\begin{equation}\label{eq.bb}
		\lambda L \varphi^*_\omega\left(\frac{t}{\lambda L}\right)+t\leq \lambda \varphi^*_\omega\left(\frac{t}{\lambda}\right)+\lambda L,
		\end{equation}
		where $L\geq 1$ is the costant appearing in formula $(\ref{l})$.
		\end{enumerate}
	\end{lem}

\begin{rem}\label{R.Pfistar} Lemma \ref{L.Pfistar} 
	 allows to list some properties that we will frequently use in the sequel.
\begin{enumerate}
		\item For every $m,M\in\mathbb{N}$ with $M\geq mL$, where $L$ is the constant appearing in formula $(\ref{l})$, and for every $\alpha \in \mathbb{N}^N_0$
		\begin{equation}\label{firstprop}
		2^{|\alpha|}\exp\left(M\varphi^*_\omega\left(\frac{|\alpha|}{M}\right)\right)\leq C \exp\left(m\varphi^*_\omega\left(\frac{|\alpha|}{m}\right)\right),
		\end{equation}
		with $C:=e^{mL}$.
		\item For every $\alpha,\beta \in\mathbb{N}^N_0$ and $m\in\mathbb{N}$
		\begin{equation}\label{secondprop}
		2m\varphi^*_\omega\left(\frac{|\alpha+\beta|}{2m}\right)\leq m\varphi^*_\omega\left(\frac{|\alpha|}{m}\right)+m\varphi^*_\omega\left(\frac{|\beta|}{m}\right)\leq m\varphi^*_\omega\left(\frac{|\alpha+\beta|}{m}\right).
		\end{equation}
	\end{enumerate}
\end{rem}

We now introduce the ultradifferentiable function spaces and their duals of Beurling type in the sense of Braun, Meise and Taylor \cite{BMT}. 

\begin{defn}\label{D.Beurling} Let $\omega$ be a non-quasianalytic weight.

	 (a) For a compact subset $K$ of $\R^N$ and $\lambda >0$ define
	\[
	\cE_{\omega,\lambda}(K):=\left\{f\in C^\infty(K)\colon p_{K,\lambda}(f):=\sup_{x\in K}\sup_{\alpha\in\N_0^N}|\partial^{\alpha}f(x)|\exp\left(-\lambda \varphi^*_\omega\left(\frac{|\alpha|}{\lambda}\right)\right)<\infty\right\}.
	\]
	Then $(\cE_{\omega,\lambda}(K), p_{K,\lambda})$ is a Banach space.
	
	 (b) For an open subset $\Omega$ of $\R^N$ define
	\[
	\cE_\omega(\Omega):=\left\{f\in C^\infty(\Omega)\colon p_{K.m}(f)<\infty\  \forall K \Subset \Omega,\, m\in\N  \right\}
	\]
	and endow it with its natural Fr\'echet space topology, i.e., with  the lc-topology generated by the system of seminorms $\{p_{K,m}\}_{K\Subset \Omega, m\in\N}$.
	The elements of $\cE_\omega(\Omega)$ are called \textit{$\omega$-ultradifferentiable functions
	of Beurling  type} on $\Omega$. The dual $\cE'_\omega(\Omega)$ of $\cE_\omega(\Omega)$ is endowed with its strong topology. 

 (c) For a compact subset $K$ of $\R^N$ define
\[
\cD_\omega(K):=\left\{f\in \cE_\omega(\R^N)\colon {\rm supp}\, f\su  K\right\}
\]
and endow it with the Fr\'echet space topology generated by the sequence $\{p_{K, m}\}_{m\in\N}$ of norms. For an open subset $\Omega$ of $\R^N$ define
\[
\cD_\omega(\Omega):=\ind_{j\rightarrow}\cD_{\omega}(K_j),
\]
where $\{K_j\}_{j\in\N}$ is any fundamental sequence of compact subsets of $\Omega$. The elements of $\cD_\omega(\Omega)$ are called \textit{test functions of Beurling type} on $\Omega$. The dual $\cD'_\omega(\Omega)$ of $\cD_\omega(\Omega)$ is endowed with its strong topology. The elements
of $\cD'_\omega(\Omega)$ are called \textit{$\omega$-ultradistributions of Beurling  type} on $\Omega$.
	\end{defn}

\begin{rem}\label{R.PspaziB} Let $\omega$ be a non-quasianalytic weight function. Then the following properties are satisfied.

	 (1) For every open subset $\Omega$ of $\R^N$ the space $\cE_\omega(\Omega)$ is a nuclear Fr\'echet space, see \cite[Proposition 4.9]{BMT}.
	 
		(2) For every compact subset $K$ of $\R^N$ we have $\cD_\omega(K)\not=\{0\}$ and $\cD_\omega(K)$ is a nuclear Fr\'echet space, see \cite[Remark 3.2 (1) and Corollary 3.6(2)]{BMT}.
		
	 (3) For every open subset $\Omega$ of $\R^N$ the inclusion $\cD_\omega(\Omega)\hookrightarrow \cE_\omega(\Omega)$ is continuous with dense range, see \cite[Proposition 4.7 (1)]{BMT}.
	 
	\end{rem}

We consider the following notation for the Fourier transform of a function $f\in L^1(\R^N)$:
\[
\cF(f)(\xi)=\hat{f}(\xi):=\int_{\R^N} f(x)e^{-ix\xi}\, dx,\quad \xi\in\R^N,
\]
with standard extensions to more general spaces of functions or distributions. We introduce the space of weighted rapidly decreasing functions of Beurling type as defined in  \cite[Definition 1.8.1]{B}.

\begin{defn}\label{D.S} 
Let $\omega$ be a non-quasianalytic weight function. We denote by  $\mathcal{S}_\omega(\mathbb{R}^N)$  the set of all functions $f\in L^1(\mathbb{R}^N)$ such that $f,\hat{f}\in C^\infty(\mathbb{R}^N)$ and for each $\lambda >0$ and $\alpha\in\N_0^N$ we have
\begin{equation}\label{cond Sw 1}
\| \exp(\lambda\omega)\partial^\alpha f\|_\infty <\infty\ \ {\rm and }\ \  
\|\exp(\lambda\omega)\partial^\alpha\hat{f}\|_\infty <\infty \; .
\end{equation}
The elements of  $\mathcal{S}_\omega(\mathbb{R}^N)$ are called \textit{$\omega$-ultradifferentiable rapidly decreasing functions of Beurling type}. 
\end{defn}

The space $\cS_\omega(\R^N)$ is a Fr\'echet space with respect to the lc-topology generated by the sequence of norms
\[
\| \exp(n\omega)\partial^\alpha f\|_\infty +\|\exp(n\omega)\partial^\alpha\hat{f}\|_\infty,\quad f\in \cS_\omega(\R^N),\ n\in\N.
\]
We denote by  $\cS'_{\omega}(\R^N)$ the dual of $\cS_{\omega}(\R^N)$  endowed with its strong topology.

\begin{rem}\label{in s} Let $\omega$ be a non-quasianalytic  weight function. Then the following properties are satisfied.
	
(1)	The condition $(\gamma)$ of Definition \ref{D.weight} implies that  $\mathcal{S}_\omega(\mathbb{R}^N)\su \mathcal{S}(\mathbb{R}^N)$  with continuous inclusion. 
	Accordingly,  we can rewrite the definition of $\mathcal{S}_\omega(\mathbb{R}^N)$ as the set of all the rapidly decreasing functions that satisfy the condition  $(\ref{cond Sw 1})$ in Definition \ref{D.S}.

(2) The Fourier transform $\mathcal{F}:\mathcal{S}_{\omega}(\mathbb{R}^N)\to \mathcal{S}_{\omega}(\mathbb{R}^N)$ is a continuous isomorphism, that can be extended in the usual way to $\cS'_\omega(\R^N)$, see \cite[Proposition 1.8.2]{B}.
	 
	(3)	The space $\mathcal{S}_\omega(\mathbb{R}^N)$ is closed under convolution, under multiplication, translation and modulation, where the translation and modulation operators are defined by $\tau_y f(x):=f(x-y)$ and $M_t f(x):=e^{itx} f(x)$, respectively, where $t,x,y\in\mathbb{R}^N$, see \cite[Propositions 1.8.3 and 18.5]{B}.
	
	(4) The inclusion $\cD_\omega(\R^N)\hookrightarrow \cS_\omega(\R^N)\hookrightarrow \cE_\omega(\R^N)$ are continuous with dense range, see \cite[Proposition 4.7.(1)]{BMT} and \cite[Propositions 1.8.6 and 1.8.7]{B}. Therefore, $\cE'_\omega(\R^N)\subset \cS_\omega'(\R^N)\subset \cD'_\omega(\R^N)$.
	
	(5) The space $\cS_\omega(\R^N)$ is a nuclear Fr\'echet space, see \cite[Theorem 3.3]{BJOS}.
\end{rem}

The space $\cS_\omega(\R^N)$ is a Fr\'echet space with different equivalent systems of seminorms. Indeed, the following result holds
(see \cite[Theorem 4.8]{BJOR} and \cite[Theorems 2.6]{BJO}).

\begin{prop}\label{P.norme}
	Let $\omega$ be a non-quasianalytic weight function and consider $f \in \mathcal{S}(\mathbb{R}^N)$. Then $f\in\cS_\omega(\R^N)$ if and only if one of the following  conditions is satisfied.
\begin{enumerate}	
\item	\begin{enumerate}	
	\item[{\rm (i)} ] $\forall\lambda>0, \; \alpha\in\mathbb{N}^N_0, \; 1\leq p\leq \infty$ $\exists C_{\alpha,\lambda,p}>0$ such that 
			$\|\exp(\lambda\omega)\partial^\alpha f\|_p\leq C_{\alpha,\lambda,p}$, and
	
\item[{\rm (ii)}] $\forall\lambda>0, \; \alpha\in\mathbb{N}^N_0,\; 1\leq p\leq \infty$  $\exists C_{\alpha,\lambda,p}>0$ such that 	
		$\| \exp(\lambda\omega)\partial^\alpha \hat{f}\|_p\leq C_{\alpha,\lambda,p} $.
			\end{enumerate}
		\item \begin{enumerate}
			\item[\rm (i)] $\forall \lambda>0,\; 1\leq p \leq \infty$ $\exists C_{\lambda,p}>0$ such that
		$\| \exp(\lambda\omega) f\|_p \leq C_{\lambda,p}$, and
		\item[\rm (ii)]	$\forall \lambda>0,\; 1\leq p \leq \infty$ $\exists C_{\lambda,p}>0$ such that
	$\|\exp(\lambda\omega)\hat{f}\|_p\leq C_{\lambda,p} $.
		\end{enumerate}
		\item $\forall\lambda,\;\mu>0,\; 1\leq p\leq \infty$ $\exists C_{\lambda,\mu,p}>0$ such that
		\[
	\underset{\alpha \in \mathbb{N}^N_0}{\sup}\| \exp(\mu\omega) \partial^\alpha f\|_p\exp\left(-\lambda\varphi^*_\omega\left(\frac{|\alpha|}{\lambda}\right)\right) \leq C_{\lambda,\mu,p} 	.\]
 \end{enumerate}
\end{prop}

\begin{rem}
	We observe that the assumption $f\in \cS(\R^N)$ in Proposition \ref{P.norme} can be replaced by the weaker assumption  $f\in C^\infty(\R^N)$.  
	Indeed, the condition $(\gamma)$ in Definition \ref{D.weight} implies for every $x\in\R^N$ and $\alpha\in\N_0^N$ that 
	\[
	|x^\alpha|\leq \exp\left(-\frac{|\alpha|}{b}\right) \exp\left(\frac{|\alpha|}{b}\omega(x)\right),
	\]
	where $b$ is the constant appearing in condition $(\gamma)$. Therefore, if one (and hence all) of the equivalent conditions (1)$\div$(3) of Proposition \ref{P.norme} is satisfied, then for every $\alpha,\beta\in\N_0^N$ and $1\leq p\leq \infty$ we have
		\[
\|x^\alpha\partial^\beta f\|_p\leq \exp\left(-\frac{|\alpha|}{b}\right) \left\|\exp\left(\frac{|\alpha|}{b}\omega\right)\partial^\beta f\right\|_p<\infty.
\]
Accordingly, $f\in \cS(\R^N)$.
\end{rem}

In the following,  we will use this system of norms generating the Fr\'echet topology of $\cS_\omega(\R^N)$:
\begin{align*}
q_{\lambda,\mu}(f):=\underset{\alpha \in \mathbb{N}^N_0}{\sup}\,\underset{x \in \mathbb{R}^N}{\sup}\, \exp\left(-\lambda\varphi^*_\omega\left(\frac{|\alpha|}{\lambda}\right)\right)\exp(\mu\omega(x)) |\partial^\alpha f(x)|, \ \lambda,\mu>0, \ f\in 	\cS_\omega(\R^N).
\end{align*}
In particular, the Fr\'echet topology of $\cS_\omega(\R^N)$ is generated by the sequence of norms $\{q_{m,n}\}_{m,n\in\mathbb{N}}$.

\section{The spaces $\cO_{M,\omega}(\R^N)$ and $\cO_{C,\omega}(\R^N)$}

We first introduce the spaces $\cO_{M,\omega}(\R^N)$ and $\cO_{C,\omega}(\R^N)$ and then we collect some  basic properties about them.

\begin{defn}\label{D.spaziO}
	Let $\omega$ be a non-quasianalytic weight function. 
	
	(a)  For $m\in \mathbb{N}$ and $n\in \mathbb{Z}$ we define the space $\mathcal{O}^m_{n,\omega}(\mathbb{R}^N)$ as the set of all functions $f\in C^\infty(\mathbb{R}^N)$ satisfying the following condition:
	\begin{equation}\label{omn}
	r_{m,n}(f):= \underset{\alpha\in \mathbb{N}^N_0}{\sup}\underset{x\in\mathbb{R}^N}{\sup}\, |\partial^\alpha f(x)|\exp\left(-n\omega(x)-m\varphi^*_\omega\left(\frac{|\alpha|}{m}\right)\right)<\infty.
\end{equation}	

(b)	We denote by $\mathcal{O}_{M,\omega}(\mathbb{R}^N)$  the set of all functions $f\in {C}^\infty(\mathbb{R}^N)$ such that for each $m\in \mathbb{N}$ there exist $C>0$ and $n \in\mathbb{N}$ such that for every $\alpha \in \mathbb{N}^N_0$ and $x\in \mathbb{R}^N$ we have
\begin{equation}\label{m}
|\partial^\alpha f(x)|\leq C\exp \left(n\omega(x)+m\varphi^*_\omega\left(\frac{|\alpha|}{m}\right)\right);		
\end{equation}
or equivalently,
\begin{equation}\label{eq.OM}
\cO_{M,\omega}(\R^N):=\bigcap_{m=1}^{\infty}\bigcup_{n=1}^\infty \mathcal{O}^m_{n,\omega}(\mathbb{R}^N).
\end{equation}
The elements of  $\mathcal{O}_{M,\omega}(\mathbb{R}^N)$ are called \textit{slowly increasing functions of Beurling type}.

(c)	We denote by  $\cO_{C,\omega}(\mathbb{R}^N)$  the set of all functions $f\in C^\infty(\R^N)$ for which there exists $n\in \mathbb{N}$ such that for every $m \in\mathbb{N}$ there exists $C>0$ so that for every $\alpha \in \mathbb{N}^N_0$ and $x\in \mathbb{R}^N$ we have
\begin{equation}\label{c}
|\partial^\alpha f(x)|\leq C\exp \left(n\omega(x)+m\varphi^*_\omega\left(\frac{|\alpha|}{m}\right)\right);		
\end{equation}
or equivalently,
	\begin{equation}\label{eq.OC}
\cO_{C,\omega}(\R^N):=\bigcup_{n=1}^{\infty}\bigcap_{m=1}^\infty \mathcal{O}^m_{n,\omega}(\mathbb{R}^N).
\end{equation}
The elements of $\mathcal{O}_{C,\omega}(\mathbb{R}^N)$ are called  \textit{very slowly increasing functions of Beurling type}.
\end{defn}
\begin{rem}\label{R.spaziD} By Definition \ref{D.spaziO}, we immediately obtain the   following properties.
	
(1)	$\cD_\omega(\R^N)\su\cO_{C,\omega}(\R^N)$ and $\cD_\omega(\R^N)\su \cO_{M,\omega}(\R^N)$.

(2)	 For every $\alpha\in \N_0^N$ the function $x^\alpha\in \cO_{M,\omega}(\R^N)$. Indeed, fixed any $\alpha\in\N_0^N$, we have that $\partial^\beta x^\alpha=\prod_{i=1}^N\alpha_i(\alpha_i-1).\ldots . (\alpha_i-\beta_i+1)x_i^{\alpha_i-\beta_i}$ for $\beta \leq \alpha$, and that $\partial^\beta x^\alpha=0$ for $\beta>\alpha$. Since condition $(\gamma)$ implies that $(1+t)\leq e^{-\frac{a}{b}} e^{\frac{1}{b}\omega(t)}$ for all $t\geq 0$ and $\omega$ is an increasing function on $[0,\infty)$, we have for every $\beta\leq \alpha$ and $x\in\R^N$ that 
\[
|\partial^\beta x^\alpha|\leq \prod_{i=1}^N\alpha_i!e^{-\frac{a}{b}(\alpha_i-\beta_i)}e^{\frac{1}{b}\omega(|x_i|)}\leq \alpha!e^{-\frac{a}{b}|\alpha|}e^{\frac{|\alpha|}{b}\omega(x)}.
\]
Accordingly, $x^\alpha\in \cO_{M,\omega}(\R^N)$.

(3) The function $\log (1+|x|^2)\in \cO_{C,\omega}(\mathbb{R}^N)$ and the proof follows as in the case (2) above.
\end{rem}

We  now give an useful characterization of the  elements of the spaces $\cO_{M,\omega}(\R^N)$ and $\cO_{C,\omega}(\R^N)$.

\begin{prop}\label{P.NcharO} Let $\omega$ be a non-quasianalytic weight function. Then the following properties are satified.
	\begin{enumerate}
		\item	A function $f\in C^\infty(\R^N)$ is a slowly increasing function of Beurling type if and only if $f\in \mathcal{E}_\omega(\mathbb{R}^N)$ and for each $m\in \mathbb{N}$ there exist $C,R>0$ and $n \in\mathbb{N}$ such that  for every  $\alpha \in \mathbb{N}^N_0$ and $x\in \mathbb{R}^N$ with $|x|\geq R$ we have
		\begin{equation}\label{m2}
		|\partial^\alpha f(x)|\leq C\exp \left(n\omega(x)+m\varphi^*_\omega\left(\frac{|\alpha|}{m}\right)\right).		
		\end{equation}
		\item  A function $f\in C^\infty(\R^N)$ is a very slowly increasing function of Beurling type if and only if $f\in \mathcal{E}_\omega(\mathbb{R}^N)$ and there exists $n\in \mathbb{N}$ such that for every $m \in\mathbb{N}$ there exist $C,R>0$ so that for every $\alpha \in \mathbb{N}^N_0$ and $x\in \mathbb{R}^N$ with $|x|\geq R$ we have
		\begin{equation}
		|\partial^\alpha f(x)|\leq C\exp \left(n\omega(x)+m\varphi^*_\omega\left(\frac{|\alpha|}{m}\right)\right),		
		\end{equation}
	\end{enumerate}
\end{prop}

\begin{proof} (1) The necessity of the condition \eqref{m2} is obvious. We need only to prove that $f\in \cE_\omega(\R^N)$. So, 
	fix  a compact subset $K$ of $\mathbb{R}^N$. Then for each $m\in\mathbb{N}$ there exists $n\in\mathbb{N}$ such that for every $\alpha \in \mathbb{N}^N_0$ and $x\in K$ we have
	\begin{equation*}
	|\partial^\alpha f(x)|\leq r_{m,n}(f) \exp\left(n\omega(x)+m\varphi^*_\omega\left(\frac{|\alpha|}{m}\right)\right)\leq D\, r_{m,n}(f)\exp\left(m\varphi^*_\omega\left(\frac{|\alpha|}{m}\right)\right),
	\end{equation*}
	where $D:=\underset{y\in K}{\max}\left\{\exp(n\omega(y)) \right\}<\infty$ is a constant depending only on $n$ and $K$. Hence, it follows for every $m\in\N$ that 
	\begin{equation*}
	p_{K,m}(f)=\underset{\alpha \in \mathbb{N}^N_0}{\sup}\, \underset{x \in K}{\sup}\,|\partial^\alpha f(x)|\exp\left(-m\varphi^*_\omega\left(\frac{|\alpha|}{m}\right)\right)\leq D\, r_{m,n}(f)<\infty.
	\end{equation*}
	Since   $K$ is arbitrary, this implies that $f\in \mathcal{E}_ {\omega}(\mathbb{R}^N)$.

	We now prove the 
	sufficiency of the condition.
	
	Fix $m \in \mathbb{N}$. By assumption there exist $C,R>0$ and $n\in \mathbb{N}$ such that the inequality  $(\ref{m2})$ is satisfied and $f\in \mathcal{E}_\omega(\mathbb{R}^N)$. Accordingly, we have 
	\[K_1:=\underset{\alpha \in \mathbb{N}^N_0}{\sup}\, \underset{|x|\leq R}{\sup}\, |\partial^\alpha f(x)|\exp\left(-m\varphi_\omega^*\left(\frac{|\alpha|}{m}\right)\right)<\infty.
	\]
	If we set $ K_2:=\underset{|x|\leq R}{\min} \exp(n\omega(x))\geq 1$, 
	then we obtain for every $ |x|\leq R$ and $ \alpha \in \mathbb{N}^N_0$ that 
	\begin{align*}
	|\partial^\alpha f(x)|&\leq K_1 \exp\left(m\varphi^*_\omega\left(\frac{|\alpha|}{m}\right)\right) =\frac{K_1}{K_2}K_2 \exp\left(m\varphi^*_\omega\left(\frac{|\alpha|}{m}\right)\right)\\
	&\leq  \frac{K_1}{K_2} \exp(n\omega(x)) \exp\left(m\varphi^*_\omega\left(\frac{|\alpha|}{m}\right)\right).
	\end{align*}
	Therefore, setting  $C':= \max \left\{\frac{K_1}{K_2}, C\right\}$, the inequality $(\ref{m})$ follows.
	
	(2) follows by the same arguments for the proof of property (1) above.
\end{proof}

We now establish some features concerning topological properties of the spaces introduced in Definition \ref{D.spaziO}.

\begin{prop}\label{P.prop} Let $\omega$ be a non-quasianalytic weight function. Then the following properties are satisfied.
	\begin{enumerate}
		\item Let $m\in \mathbb{N}$ and $n\in \mathbb{Z}$. Then  $\left(\mathcal{O}^m_{n,\omega}(\mathbb{R}^N), r_{m,n}\right)$ is a Banach space.
		\item For every $n, n'\in \mathbb{Z}$ with $n\leq n'$ and $m\in \mathbb{N}$, the inclusion
		\begin{equation}\label{inc1}
		(\mathcal{O}^m_{n,\omega}(\mathbb{R}^N),r_{m,n}) \hookrightarrow (\mathcal{O}^m_{n',\omega}(\mathbb{R}^N),r_{m,n'})
		\end{equation}
		is well-defined and continuous.
		\item  For every $n\in \mathbb{Z}$ and $m, m'\in \mathbb{N}$ with $m\leq m'$, the inclusion
		\begin{equation}\label{inc2}
		(\mathcal{O}^{m'}_{n,\omega}(\mathbb{R}^N), r_{m',n})\hookrightarrow (\mathcal{O}^m_{n,\omega}(\mathbb{R}^N), r_{m,n})
		\end{equation}
		is well-defined and continuous.
		\end{enumerate}
	\end{prop}

\begin{proof} (1) Fix $m\in\N$ and $n\in\Z$. 	It sufficies to show only the completeness. So, we fix  a Cauchy sequence  $\{f_j\}_{j\in\N}$ in $(\mathcal{O}^m_{n,\omega}(\mathbb{R}^N), r_{m,n})$ and observe that
	\begin{equation}\label{dsb}
	|\partial^\alpha f_j(x)-\partial^\alpha f_{j'}(x)|\leq \exp\left(n\omega(x)+m\varphi^*_\omega\left(\frac{|\alpha|}{m}\right)\right) r_{m,n}(f_j-f_{j'}),
	\end{equation}
	for all $j,j'\in \mathbb{N}$,  $\alpha\in \mathbb{N}^N_0$ and  $x\in \mathbb{R}^N$. Therefore, for any compact subset $K$ of $\mathbb{R}^N$, we have
	\begin{equation*}
	\underset{\beta\leq\alpha}{\sup}\,\underset{x\in K}{\sup}\, |\partial^\beta f_j(x)-\partial^\beta f_{j'}(x)|\leq C_{K,\alpha}r_{m,n}(f_j-f_{j'}),
	\end{equation*}
	for all $j,j'\in \mathbb{N}$ and $\alpha\in \mathbb{N}^N_0$, where $C_{K,\alpha}:=\underset{x\in K}{\sup}\exp(n\omega(x))\, \underset{\beta\leq\alpha}{\sup}\, \exp\left(m\varphi^*_\omega\left(\frac{|\alpha|}{m}\right)\right)<\infty$. Since $r_{m,n}(f_j-f_{j'})\to 0$ as $j,j'\to \infty$, and $C^\infty(\mathbb{R}^N)$ is a Fr\'echet space, it follows that $\{f_j\}_{j\in\N}$ is also a Cauchy sequence in $C^\infty(\mathbb{R}^N)$. Accordingly, there exists $f\in C^\infty(\mathbb{R}^N)$ such that $f_j \to f$ in $C^\infty(\mathbb{R}^N)$, as $j\to \infty$. In particular, $\partial^\alpha f_j \to \partial^\alpha f$ uniformly on compact subsets of $\mathbb{R}^N$ for every $\alpha \in \mathbb{N}^N_0$.
	\\[0.1cm]
	Now, fix $\epsilon>0$. Since $\{f_j\}_{j\in\N}$ is a Cauchy sequence in $(\mathcal{O}^m_{n,\omega}(\mathbb{R}^N),r_{m,n})$, there exists $j_0\in \mathbb{N}$ such that $	r_{m,n}(f_j-f_{j'})\leq\epsilon$ for all $j,j'\geq j_0$.
It follows by  $(\ref{dsb})$ that  for every $j,j'\geq j_0$,  $x\in \mathbb{R}^N$ and $\alpha\in\mathbb{N}^N_0$ we have
	\begin{equation*}
	|\partial^\alpha f_j(x)-\partial^\alpha f_{j'}(x)|\leq \epsilon \exp\left(n\omega(x)+m\varphi^*_\omega\left(\frac{|\alpha|}{m}\right)\right).
	\end{equation*}
	Letting $j'\to \infty$ in the inequality above, we obtain for every $j\geq j_0$,   $x\in \mathbb{R}^N$ and $\alpha\in\mathbb{N}^N_0$ that 
	\begin{equation*}
	|\partial^\alpha f_j(x)-\partial^\alpha f(x)|\leq \epsilon \exp\left(n\omega(x)+m\varphi^*_\omega\left(\frac{|\alpha|}{m}\right)\right).
	\end{equation*}
	Accordingly, we have 
	\begin{align*}
	&r_{m,n}(f_j-f)\leq\epsilon\quad \forall j\geq j_0,\\&r_{m,n}(f)\leq r_{m,n}(f-f_{j_0})+r_{m,n}(f_{j_0})\leq \epsilon + r_{m,n}(f_{j_0}) <\infty.
	\end{align*}
	This means that $f\in \mathcal{O}^m_{n,\omega}(\mathbb{R}^N)$ and that $f_j\to f$ in $(\mathcal{O}^m_{n,\omega}(\mathbb{R}^N), r_{m,n})$ for $j \to \infty$, as $\epsilon>0$ is arbitrary.
	
	(2)  Fix $n,n'\in \mathbb{Z}$ with $n\leq n'$ and $ m \in\mathbb{N}$. Then for every $f\in \mathcal{O}^m_{n,\omega}(\mathbb{R}^N)$, we have
	\begin{align*}
	r_{m,n'}(f)&=\underset{\alpha\in \mathbb{N}^N_0}{\sup}\underset{x\in\mathbb{R}^N}{\sup}\, |\partial^\alpha f(x)|\exp\left(-n'\omega(x)-m\varphi^*_\omega\left(\frac{|\alpha|}{m}\right)\right)\\& \leq \underset{\alpha\in \mathbb{N}^N_0}{\sup}\underset{x\in\mathbb{R}^N}{\sup}\, |\partial^\alpha f(x)|\exp\left(-n\omega(x)-m\varphi^*_\omega\left(\frac{|\alpha|}{m}\right)\right)=r_{m,n}(f).
	\end{align*}
	Therefore, the inclusion in $(\ref{inc1})$ is well-defined and continuous.

(3) Fix $m,m'\in\mathbb{N}$ with $m\leq m'$ and $n\in\mathbb{Z}$. Since $\varphi^*_\omega(t)/t$ is an increasing function in $(0,\infty)$ (see Lemma \ref{L.Pfistar}(2)),  for every $f\in \mathcal{O}^{m'}_{n,\omega}(\mathbb{R}^N)$ we have
	\begin{align*}
	r_{m,n}(f)&=\underset{\alpha\in \mathbb{N}^N_0}{\sup}\underset{x\in\mathbb{R}^N}{\sup}\, |\partial^\alpha f(x)|\exp\left(-n\omega(x)-m\varphi^*_\omega\left(\frac{|\alpha|}{m}\right)\right)\\& \leq \underset{\alpha\in \mathbb{N}^N_0}{\sup}\underset{x\in\mathbb{R}^N}{\sup}\, |\partial^\alpha f(x)|\exp\left(-n\omega(x)-m'\varphi^*_\omega\left(\frac{|\alpha|}{m'}\right)\right)=r_{m',n}(f).
	\end{align*}
	Therefore, the inclusion in $(\ref{inc2})$ is well-defined and continuous.
\end{proof}

Via Proposition \ref{P.prop}(1)-(2) we deduce that the sequence $\{(\mathcal{O}^m_{n,\omega}(\mathbb{R}^N),r_{m,n})\}_{n\in\mathbb{N}}$ of Banach spaces forms for each $m\in\mathbb{N}$ an inductive spectrum. The space $\bigcup_{n=1}^{\infty}\mathcal{O}^{m}_{n,\omega}(\mathbb{R}^N)$, endowed  with the inductive topology defined by $\{(\mathcal{O}^m_{n,\omega}(\mathbb{R}^N),r_{m,n})\}_{n\in\mathbb{N}}$, is then an (LB)-space for each $m\in\N$. 
On the other hand, via Proposition \ref{P.prop}(1)-(3) we also deduce  for every $n\in\mathbb{N}$ that  the sequence $\{(\mathcal{O}^m_{n,\omega}(\mathbb{R}^N),r_{m,n})\}_{m\in\mathbb{N}}$ of Banach spaces forms a projective spectrum. So, for every $n\in\mathbb{N}$ the space $\bigcap_{m=1}^{\infty}\mathcal{O}^{m}_{n,\omega}(\mathbb{R}^N)$, endowed with the projective topology defined by $\{(\mathcal{O}^m_{n,\omega}(\mathbb{R}^N),r_{m,n})\}_{m\in\mathbb{N}}$, is a  Fr\'echet space. In the following we always suppose that the spaces
 $\bigcup_{n=1}^{\infty}\mathcal{O}^{m}_{n,\omega}(\mathbb{R}^N)$  and
 $\bigcap_{m=1}^{\infty}\mathcal{O}^{m}_{n,\omega}(\mathbb{R}^N)$ are equipped respectively with the (LB)-topology and  Fr\'echet topology defined above. 
 In particular, the spaces $\bigcup_{n=1}^{\infty}\mathcal{O}^{m}_{n,\omega}(\mathbb{R}^N)$ and $\bigcap_{m=1}^{\infty}\mathcal{O}^{m}_{n,\omega}(\mathbb{R}^N)$ satisfy the following properties.
 
 \begin{prop}\label{P.LBcomplete} 	Let $\omega$ be a non-quasianalytic weight function and $m\in\N$. Then $\bigcup_{n=1}^{\infty}\mathcal{O}^{m}_{n,\omega}(\mathbb{R}^N)$ is a complete (LB)-space.
 	 \end{prop}
  
  \begin{proof} In order to show the completeness of the space $\bigcup_{n=1}^{\infty}\mathcal{O}^{m}_{n,\omega}(\mathbb{R}^N)$, we first prove that the inclusion  $\bigcup_{n=1}^{\infty}\mathcal{O}^{m}_{n,\omega}(\mathbb{R}^N)\hookrightarrow C^\infty(\R^N)$ is continuous as follows.
  	
  	Fix $n\in\N$. Then for every compact subset  $K$ of $\R^N$ and $\alpha\in \N_0^N$ we have
  	\begin{align*}
  	&\sup_{x\in K }\sup_{\beta\leq \alpha}|\partial^\beta f(x)|=\\
  	&\quad =\sup_{x\in K }\sup_{\beta\leq \alpha}|\partial^\beta f(x)|\exp\left(-n\omega(x)-m\varphi_\omega^*\left(\frac{|\beta|}{m}\right)\right)\exp\left(n\omega(x)+m\varphi_\omega^*\left(\frac{|\beta|}{m}\right)\right)\\
  	&\quad \leq C_{K,\alpha}r_{m,n}(f)
  	\end{align*}
  	for each $f\in \cO_{n,\omega}^m(\R^N)$, where $C_{K,\alpha}:=\sup_{x\in K}\sup_{\beta\leq \alpha}\exp\left(n\omega(x)+m\varphi_\omega^*\left(\frac{|\beta|}{m}\right)\right)<\infty$ is a positive constant depending on $K$ and $\alpha$.  This means that the inclusion $(\cO_{n,\omega}^m(\R^N),r_{m,n})\hookrightarrow C^\infty(\R^N)$ is continuous. Since $n\in\N$ is arbitrary and $\bigcup_{n=1}^{\infty}\mathcal{O}^{m}_{n,\omega}(\mathbb{R}^N)$ is an (LB)-space, the inclusion $\bigcup_{n=1}^{\infty}\mathcal{O}^{m}_{n,\omega}(\mathbb{R}^N)\hookrightarrow C^\infty(\R^N)$ is continuous too. Therefore, there exists a Hausdorff lc-topology $\tau$ on $\bigcup_{n=1}^{\infty}\mathcal{O}^{m}_{n,\omega}(\mathbb{R}^N)$ with the property that the closed unit ball of each $\cO^m_{n,\omega}(\R^N)$ is relatively $\tau$-compact. But, the  closed unit ball of each $\cO^m_{n,\omega}(\R^N)$ is also $\tau$-compact. Indeed,  let $n\in\N$ and let  $\{f_j\}_{j\in N}\su \{g\in \cO^m_{n,\omega}(\R^N)\colon r_{m,n}(g)\leq 1\}$  $\tau$-convergent to some $f\in C^\infty(\R^N)$. Then for every $x\in \R^N$, $j\in\N$ and $\alpha\in\N_0^N$ we have
  	\begin{equation}\label{eq.limite}
  	|\partial^\alpha f_j(x)|\leq \exp\left(n\omega(x)+m\varphi^*_\omega\left(\frac{|\alpha|}{m}\right)\right).
  	\end{equation}
  	Since $f_j\to f$ in $C^\infty(\R^N)$ as $j\to\infty$ and hence $\partial^\alpha f_j\to \partial^\alpha f$ pointwise on $\R^N$ for each $\alpha\in\N_0^N$, it follows by letting $j\to\infty$ in \eqref{eq.limite} for every $x\in \R^N$ and $\alpha\in\N_0^N$ that 
  	\[
  	|\partial^\alpha f(x)|\leq \exp\left(n\omega(x)+m\varphi^*_\omega\left(\frac{|\alpha|}{m}\right)\right).
  	\]
  	This implies that $r_{m,n}(f)\leq 1$. So, $\{g\in \cO^m_{n,\omega}(\R^N)\colon r_{m,n}(g)\leq 1\}$ is $\tau$-closed.
  	
  	The result now follows by Mujica \cite[Theorem 1]{Mu}.
  \end{proof}
 
\begin{prop}\label{incslb}
	Let $\omega$ be a non-quasianalytic weight function. Then the following properties are satisfied.
	\begin{enumerate}
		\item For every $n,n'\in \mathbb{Z}$ with $n\leq n'$, the inclusion
		\begin{equation}\label{inc3}
		\bigcap_{m=1}^{\infty}\mathcal{O}^{m}_{n,\omega}(\mathbb{R}^N)\hookrightarrow \bigcap_{m=1}^{\infty}\mathcal{O}^m_{n',\omega}(\mathbb{R}^N)
		\end{equation}
		is well-defined and continuous.
		\item For every $m,m'\in \mathbb{N}$ with $m\leq m'$, the inclusion
		\begin{equation}\label{inc4}
		\bigcup_{n=1}^{\infty}\mathcal{O}^{m'}_{n,\omega}(\mathbb{R}^N)\hookrightarrow \bigcup_{n=1}^{\infty}\mathcal{O}^m_{n,\omega}(\mathbb{R}^N)
		\end{equation}
		is well-defined and continuous.
	\end{enumerate}
	\end{prop}

\begin{proof} (1) Fix $n,n'\in \mathbb{Z}$ with $n\leq n'$. Then, for every $h\in\mathbb{N}$ and $f\in \bigcap_{m=1}^{\infty}\mathcal{O}^{m}_{n,\omega}(\mathbb{R}^N)$, we have
	\begin{align*}
	r_{h,n'}(f)&=\underset{\alpha\in \mathbb{N}^N_0}{\sup}\underset{x\in\mathbb{R}^N}{\sup}\, |\partial^\alpha f(x)|\exp\left(-n'\omega(x)-h\varphi^*_\omega\left(\frac{|\alpha|}{h}\right)\right)\\& \leq \underset{\alpha\in \mathbb{N}^N_0}{\sup}\underset{x\in\mathbb{R}^N}{\sup}\, |\partial^\alpha f(x)|\exp\left(-n\omega(x)-h\varphi^*_\omega\left(\frac{|\alpha|}{h}\right)\right)=r_{h,n}(f).
	\end{align*}
	This means that  the inclusion in $(\ref{inc3})$ is well-defined and continuous.
	
	(2) Fix $m,m'\in \mathbb{N}$ with $m\leq m'$ and $h\in\N$. Then by Proposition \ref{P.prop}(3) the inclusion $(\mathcal{O}^{m'}_{h,\omega}(\mathbb{R}^N),r_{m',h})\hookrightarrow (\mathcal{O}^m_{h,\omega}(\mathbb{R}^N),r_{m,h})$ is continuous. On the other hand,    the inclusion  $(\mathcal{O}^{m}_{h,\omega}(\mathbb{R}^N),r_{m,h})\hookrightarrow \bigcup_{n=1}^{\infty}\mathcal{O}^m_{n,\omega}(\mathbb{R}^N)$ is also continuous.
	 Accordingly, the inclusion
	\begin{equation*}
	(\mathcal{O}^{m'}_{h,\omega}(\mathbb{R}^N),r_{m',h})\hookrightarrow \bigcup_{n=1}^{\infty}\mathcal{O}^m_{n,\omega}(\mathbb{R}^N)
	\end{equation*}
	is continuous.  Since $h\in \mathbb{N}$ is arbitrary and  $\bigcup_{h=1}^\infty \mathcal{O}^{m'}_{h,\omega}(\mathbb{R}^N)$ is an (LB)-space, it follows that the inclusion
	\begin{equation*}
	\bigcup_{h=1}^{\infty}\mathcal{O}^{m'}_{h,\omega}(\mathbb{R}^N)\hookrightarrow \bigcup_{n=1}^{\infty}\mathcal{O}^m_{n,\omega}(\mathbb{R}^N)
	\end{equation*}
	is continuous, i.e.,  the inclusion in $(\ref{inc3})$ is well-defined and continuous.
	\end{proof}

 Proposition \ref{incslb}(1) yields that the sequence $\left\{\bigcap_{m=1}^\infty \mathcal{O}^m_{n,\omega}(\mathbb{R}^N)\right\}_{n\in\mathbb{N}}$ of Fr\'echet spaces forms an inductive spectrum. So, we can endow the space $\mathcal{O}_{C,\omega}(\mathbb{R}^N)$ with the inductive topology defined by  $\left\{\bigcap_{m=1}^\infty \mathcal{O}^m_{n,\omega}(\mathbb{R}^N)\right\}_{n\in\mathbb{N}}$.  Thereby,  $\mathcal{O}_{C,\omega}(\mathbb{R}^N)$, equipped with such an inductive topology,  is an (LF)-space.
 Proposition \ref{incslb}(2)  also yields that the sequence $\left\{\bigcup_{n=1}^\infty \mathcal{O}^m_{n,\omega}(\mathbb{R}^N)\right\}_{m\in\mathbb{N}}$  of (LB)-spaces is a projective spectrum. So, we can endow the space $\mathcal{O}_{M,\omega}(\mathbb{R}^N)$ with the projective topology defined by  $\left\{\bigcup_{n=1}^\infty \mathcal{O}^m_{n,\omega}(\mathbb{R}^N)\right\}_{m\in\mathbb{N}}$. Thereby, $\mathcal{O}_{M,\omega}(\mathbb{R}^N)$, equipped with this projective topology, is a projective limit of (LB)-spaces. By Proposition \ref{P.LBcomplete},  it is then a complete lc-space. In the following we always assume that the spaces $\cO_{C,\omega}(\R^N)$ and  $\cO_{M,\omega}(\R^N)$ are endowed with the lc-topologies defined above.

\begin{rem}\label{R.rappS} 
 We observe that, for fixed $m\in \mathbb{N}$ and $n\in\mathbb{N}$, we have that 
 \begin{equation*}
 q_{m,n}(f)=\underset{\alpha \in \mathbb{N}^N_0}{\sup}\,\underset{x \in \mathbb{R}^N}{\sup}\, \exp\left(-m\varphi^*_\omega\left(\frac{|\alpha|}{m}\right)\right)\exp(n\omega(x)) |\partial^\alpha f(x)|=r_{m,-n}(f).
 \end{equation*}
for every $f\in \cS_\omega(\R^N)$. Since $\{q_{m,n}\}_{m,n\in\N}$ is a fundamental sequence of norms generating the Fr\'echet topology   of $\cS_\omega(\R^N)$, it follows that  $\mathcal{S}_{\omega}(\mathbb{R}^N)=\bigcap_{n=1}^{\infty}\bigcap_{m=1}^\infty \mathcal{O}^m_{-n,\omega}(\mathbb{R}^N)$.
\end{rem}

\begin{thm}\label{T.incl} 	Let $\omega$ be a non-quasianalytic weight function. Then the following properties are satisfied.
	\begin{enumerate}
				\item The inclusion
		\begin{equation}\label{eq.inclOO}
		\mathcal{O}_{C,\omega}(\mathbb{R}^N) \hookrightarrow \mathcal{O}_{M,\omega}(\mathbb{R}^N)
		\end{equation}
		is well-defined and continuous.	
		\item The inclusions 
		\begin{equation}\label{e.incOC}
	\cS_\omega(\R^N)\hookrightarrow	\cO_{C,\omega}(\R^N)\hookrightarrow \cE_\omega(\R^N)
		\end{equation}
		are well-defined and continuous.
	\end{enumerate}
	\end{thm}

\begin{proof} 
		(1) Fix $f\in \mathcal{O}_{C,\omega}(\mathbb{R}^N)$ and $m\in \mathbb{N}$. Then  there exists $n\in \mathbb{N}$ such that
	$r_{m,n}(f)<\infty$ and so, $f\in \cO_{n,\omega}^m(\R^N)\subset \bigcup_{h=1}^\infty \mathcal{O}^m_{h,\omega}(\mathbb{R}^N)$. Since $m\in\N$ is arbitrary, we can conclude that 
	 $f\in \mathcal{O}_{M,\omega}(\mathbb{R}^N)$. We have so shown that the inclusion is well-defined. We now prove that such an inclusion is continuous as  follows. 

	We first observe that for every $m'\in \mathbb{N}$ and $n,n' \in \mathbb{N}$ with $n\leq n'$, the inclusion
	\begin{equation*}
	\bigcap_{m=1}^{\infty}\mathcal{O}^{m}_{n,\omega}(\mathbb{R}^N)\hookrightarrow \mathcal{O}^{m'}_{n',\omega}(\mathbb{R}^N)
	\end{equation*}
	is well-defined and continuous. Indeed, for every $f\in\bigcap_{m=1}^{\infty}\mathcal{O}^{m}_{n,\omega}(\mathbb{R}^N)$ we have
	\begin{align*}
	r_{m',n'}(f)&=\underset{\alpha\in \mathbb{N}^N_0}{\sup}\underset{x\in\mathbb{R}^N}{\sup}\, |\partial^\alpha f(x)|\exp\left(-n'\omega(x)-m'\varphi^*_\omega\left(\frac{|\alpha|}{m'}\right)\right)\\& \leq \underset{\alpha\in \mathbb{N}^N_0}{\sup}\underset{x\in\mathbb{R}^N}{\sup}\, |\partial^\alpha f(x)|\exp\left(-n\omega(x)-m'\varphi^*_\omega\left(\frac{|\alpha|}{m'}\right)\right)=r_{m',n}(f).
	\end{align*}
	Since $r_{m',n}\in \{r_{m,n}\}_{m\in\mathbb{N}}$ and $\{r_{m,n}\}_{m\in\mathbb{N}}$ is a sequence of norms defyning the lc-topology of $\bigcap_{m=1}^{\infty}\mathcal{O}^{m}_{n,\omega}(\mathbb{R}^N)$, it follows that the inclusion $\bigcap_{m=1}^{\infty}\mathcal{O}^{m}_{n,\omega}(\mathbb{R}^N)\hookrightarrow \mathcal{O}^{m'}_{n',\omega}(\mathbb{R}^N)$ is continuous.
Taking into account that $\bigcup_{h=1}^{\infty}\mathcal{O}^{m'}_{h,\omega}(\mathbb{R}^N)$ is an (LB)-space and so, $ \mathcal{O}^{m'}_{n',\omega}(\mathbb{R}^N)\hookrightarrow\bigcup_{h=1}^{\infty}\mathcal{O}^{m'}_{h,\omega}(\mathbb{R}^N)$ continuously, we deduce  for every $m'\in \mathbb{N}$ and $n\in\mathbb{N}$ that the inclusion
	\begin{equation*}
	\bigcap_{m=1}^{\infty}\mathcal{O}^{m}_{n,\omega}(\mathbb{R}^N)\hookrightarrow \bigcup_{h=1}^{\infty}\mathcal{O}^{m'}_{h,\omega}(\mathbb{R}^N)
	\end{equation*}
	is  well-defined and continuous. Since $\mathcal{O}_{M,\omega}(\mathbb{R}^N)=\bigcap_{m'=1}^\infty\bigcup_{h=1}^\infty \mathcal{O}^{m'}_{h,\omega}(\mathbb{R}^N)$ is the projective limit of the sequence $\left\{\bigcup_{h=1}^\infty \mathcal{O}^{m'}_{h,\omega}(\mathbb{R}^N)\right\}_{m'\in\mathbb{N}}$ of (LB)-spaces, it follows  for every $n\in\mathbb{N}$ that the inclusion
	\begin{equation*}
	\bigcap_{m=1}^{\infty}\mathcal{O}^{m}_{n,\omega}(\mathbb{R}^N)\hookrightarrow \mathcal{O}_{M,\omega}(\mathbb{R}^N)
	\end{equation*}
	is well-defined and continuous. Finally, since $\mathcal{O}_{C,\omega}(\mathbb{R}^N)$ is the inductive limit of the sequence $\left\{\bigcap_{m=1}^\infty \mathcal{O}^{m}_{n,\omega}(\mathbb{R}^N)\right\}_{n\in\mathbb{N}}$ of Fr\'echet space, it follows that the inclusion
	\begin{equation*}
	\mathcal{O}_{C,\omega}(\mathbb{R}^N) \hookrightarrow \mathcal{O}_{M,\omega}(\mathbb{R}^N)
	\end{equation*}
	is continuous. 
	
	(2) We first show that the inclusion $\cO_{C,\omega}(\R^N)\hookrightarrow \cE_\omega(\R^N)$ is continuous. So, fix $f\in \cO_{C,\omega}(\R^N)$ and  $n\in\mathbb{N}$. Then $f\in \bigcap_{h=1}^{\infty}\mathcal{O}^{h}_{n,\omega}(\mathbb{R}^N)$. This implies   for every $m\in\mathbb{N}$ and $K$ compact subset of $\mathbb{R}^N$ that
	\begin{align*}
	p_{K,m}(f)&=\underset{\alpha \in \mathbb{N}^N_0}{\sup}\, \underset{x \in K}{\sup}\,|\partial^\alpha f(x)|\exp\left(-m\varphi^*_\omega\left(\frac{|\alpha|}{m}\right)\right)\\&= \underset{\alpha \in \mathbb{N}^N_0}{\sup}\, \underset{x \in K}{\sup}\,|\partial^\alpha f(x)|\exp\left(-m\varphi^*_\omega\left(\frac{|\alpha|}{m}\right)\right)\exp(n\omega(x)-n\omega(x))\\&\leq D\underset{\alpha \in \mathbb{N}^N_0}{\sup}\, \underset{x \in K}{\sup}\,|\partial^\alpha f(x)|\exp\left(-m\varphi^*_\omega\left(\frac{|\alpha|}{m}\right)\right)\exp(-n\omega(x))\leq D\, r_{m,n}(f)<\infty,
	\end{align*}
	where $D:=\underset{y\in K}{\max}\left\{\exp(n\omega(y)) \right\}<\infty$ is a constant depending only on $n$ and $K$. Since $\{r_{m,n}\}_{m\in\mathbb{N}}$ is a sequence of norms generating the lc-topology of the Fr\'echet space $\bigcap_{h=1}^{\infty}\mathcal{O}^{h}_{n,\omega}(\mathbb{R}^N)$ and $\{p_{K,m}\}_{K\Subset\mathbb{R}^N}$ is a sequence of seminorms generating the lc-topology of the Fr\'echet space $\mathcal{E}_{\omega}(\mathbb{R}^N)$, it follows that the inclusion
	\begin{equation*}
	\bigcap_{h=1}^{\infty}\mathcal{O}^{h}_{n,\omega}(\mathbb{R}^N)\hookrightarrow \mathcal{E}_{\omega}(\mathbb{R}^N)
	\end{equation*}
	is continuous. Since $n\in\mathbb{N}$ is arbitrary and $\mathcal{O}_{C,\omega}(\mathbb{R}^N)$ is the inductive limit of the Fr\'echet spaces $\left\{\bigcap_{h=1}^{\infty}\mathcal{O}^{h}_{n,\omega}(\mathbb{R}^N)\right\}_{n\in\mathbb{N}}$, it follows that the inclusion on the right of \eqref{e.incOC} is continuous. 
	
We now show that the inclusion $\cS_\omega(\R^N)\hookrightarrow \cO_{C,\omega}(\R^N)$ is continuous. So,	fix $f\in \mathcal{S}_{\omega}(\mathbb{R}^N)$. Then for every $m,n\in\mathbb{N}$ we have that $q_{m,n}(f)<\infty$, thereby implying that
	\begin{align}
	r_{m,n}(f)&=\underset{\alpha \in \mathbb{N}^N_0}{\sup}\,\underset{x \in \mathbb{R}^N}{\sup}\, \exp\left(-m\varphi^*_\omega\left(\frac{|\alpha|}{m}\right)\right)\exp(-n\omega(x)) |\partial^\alpha f(x)|\nonumber\\
	& <q_{m,n}(f)<\infty.
	\end{align}
	Hence, $f\in \mathcal{O}_{C,\omega}(\mathbb{R}^N)$. This shows that the inclusion is well-defined. To prove the continuity of the inclusion, we observe that for every $n,n'\in\mathbb{N}$ and $m\in\mathbb{N}$ the inclusion
	\begin{equation*}
	\bigcap_{m=1}^{\infty} \mathcal{O}^m_{-n,\omega}(\mathbb{R}^N) \hookrightarrow \bigcap_{m=1}^{\infty}\mathcal{O}^m_{n',\omega}(\mathbb{R}^N)
	\end{equation*}
	is continuous. Since $\bigcap_{m=1}^{\infty}\mathcal{O}^m_{n',\omega}(\mathbb{R}^N)\hookrightarrow \mathcal{O}_{C,\omega}(\mathbb{R}^N)$ continuously for all $n'\in\mathbb{N}$, it follows  for every $m\in\mathbb{N}$ that the inclusion
	\begin{equation*}
	\bigcap_{m=1}^{\infty} \mathcal{O}^m_{-n,\omega}(\mathbb{R}^N) \hookrightarrow \mathcal{O}_{C,\omega}(\mathbb{R}^N)
	\end{equation*}
	is continuous. Accordingly, as $\mathcal{S}_{\omega}(\mathbb{R}^N)$ is the projective limit of the Fr\'echet spaces $\left\{\bigcap_{m=1}^{\infty} \mathcal{O}^m_{-n,\omega}(\mathbb{R}^N)\right\}_{n\in\mathbb{N}}$, we can conclude that the inclusion on the left of  \eqref{eq.inclOO} is continuous.
	\end{proof}

Finally, $\cO_{M,\omega}(\R^N)$ and $\cO_{C,\omega}(\R^N)$ have the following important property.

\begin{thm}\label{den}
	Let $\omega$ be a non-quasianalytic weight function. Then the following  inclusions
	\begin{equation*}
	\mathcal{D}_{\omega}(\mathbb{R}^N) \su \mathcal{S}_{\omega}(\mathbb{R}^N) \su \mathcal{O}_ {C,\omega}(\mathbb{R}^N)\su \mathcal{O}_ {M,\omega}(\mathbb{R}^N)\su \mathcal{E}_ {\omega}(\mathbb{R}^N)
	\end{equation*}
	are dense.
\end{thm}

\begin{proof}
	The space  $\mathcal{D}_{\omega}(\mathbb{R}^N)$ is a dense subspace  of  $\mathcal{S}_{\omega}(\mathbb{R}^N)$ and of $\mathcal{E}_{\omega}(\mathbb{R}^N)$, see \cite[Proposition 4.7.(1)]{BMT} and \cite[Propositions 1.8.6 and 1.8.7]{B}. So, both the spaces $\mathcal{O}_{C,\omega}(\mathbb{R}^N)$ and $\mathcal{O}_{M,\omega}(\mathbb{R}^N)$ are  dense subspaces of $\mathcal{E}_{\omega}(\mathbb{R}^N)$.
	
		We now  prove that $\mathcal{S}_{\omega}(\mathbb{R}^N)$ is  a dense subspace of $\mathcal{O}_{C,\omega}(\mathbb{R}^N)$ as follows. 
		
		Since $\mathcal{D}_{\omega}(\mathbb{R}^N)\su\mathcal{S}_{\omega}(\mathbb{R}^N) $, it suffices to show only that $\mathcal{D}_{\omega}(\mathbb{R}^N)$ is a  dense subspace of  $\mathcal{O}_{C,\omega}(\mathbb{R}^N)$. So, fix $f\in \mathcal{O}_{C,\omega}(\mathbb{R}^N)$
 and $\phi\in \mathcal{D}_{\omega}(\mathbb{R}^N)$ such that $\phi \equiv 1$ on $\overline{B}_1(0)$ and $0\leq \phi\leq 1$. Then for every $\epsilon >0$ the function $\phi_\epsilon (x):=\phi(\epsilon x)f(x) $, for $x\in\R^N$, belongs to  $ \mathcal{D}_{\omega}(\mathbb{R}^N)$ because $\mathcal{O}_{C,\omega}(\mathbb{R}^N)\su\mathcal{E}_{\omega}(\mathbb{R}^N)$. Since $f\in \mathcal{O}_{C,\omega}(\mathbb{R}^N)$, we have that $f\in \bigcap_{m=1}^{\infty} \mathcal{O}^m_{n,\omega}(\mathbb{R}^N)$ for some $n\in \mathbb{N}$. This implies that  $f\in \bigcap_{m=1}^{\infty} \mathcal{O}^m_{h,\omega}(\mathbb{R}^N)$ for every $h\in \mathbb{N}$ with $h\geq n$. We claim   that 
 $r_{m,n+1}(\phi_\epsilon-f)\to 0$  as $\epsilon\to 0^+$ for every $m\in\mathbb{N}$. To show the claim, we proceed as follows.
	
	Fix $m\in\mathbb{N}$ and let $M\in\mathbb{N}$ such that $M\geq Lm$, where $L\geq 1$ is the costant appearing in formula $(\ref{l})$. Then for every $\alpha\in\mathbb{N}^N_0$, $x\in\mathbb{R}^N$ and $\epsilon>0$ we have
	\begin{align}\label{eq.densC}
		|\partial^\alpha \phi_\epsilon(x)-\partial^\alpha f(x)|&\leq |\partial^\alpha f(x)(\phi(\epsilon x)-1)|+\sum_{\beta<\alpha} \binom{\alpha}{\beta} |\partial^\beta f(x)|\epsilon^{\alpha-\beta}|\partial^{\alpha-\beta} \phi (\epsilon x)|\nonumber\\
		& \leq |\partial^\alpha f(x)(\phi(\epsilon x)-1)|+\sum_{\beta<\alpha} \binom{\alpha}{\beta} \epsilon^{\alpha-\beta}r_{M,n+1}(f)\exp((n+1)\omega(x))\times\nonumber\\
		&\quad \times \exp\left(M\varphi^*_\omega\left(\frac{|\beta|}{M}\right)\right) r_{M,0}(\phi)\exp\left(M\varphi^*_\omega\left(\frac{|\alpha-\beta|}{M}\right)\right)\nonumber\\
		&\leq|\partial^\alpha f(x)(\phi(\epsilon x)-1)|+\sum_{\beta<\alpha} \binom{\alpha}{\beta} \epsilon^{\alpha-\beta}r_{M,n+1}(f)\exp((n+1)\omega(x))\times\nonumber \\
		&\quad \times r_{M,0}(\phi)\exp\left(M\varphi^*_\omega\left(\frac{|\alpha|}{M}\right)\right)\nonumber \\
		& \leq|\partial^\alpha f(x)(\phi(\epsilon x)-1)|\nonumber\\&+r_{M,n+1}(f)r_{M,0}(\phi)\exp((n+1)\omega(x))\exp\left(M\varphi^*_\omega\left(\frac{|\alpha|}{M}\right)\right)\epsilon 2^{|\alpha|},
	\end{align}
after having observed  that  $\sum_{\beta<\alpha} \binom{\alpha}{\beta} \epsilon^{\alpha-\beta}\leq \epsilon 2^{|\alpha|}$ for every $\alpha\in\N_0^N$ and using formula (\ref{secondprop}). Since $M\geq mL$, applying the inequality $(\ref{firstprop})$ it follows via \eqref{eq.densC} for every $\alpha\in\N_0^N$, $x\in\R^N$  and $\epsilon>0$ that 
		\begin{align*}
		|\partial^\alpha \phi_\epsilon(x)-\partial^\alpha f(x)|&\leq |\partial^\alpha f(x)(\phi(\epsilon x)-1)|\\
		&+\epsilon C r_{M,n+1}(f)r_{M,0}(\phi)\exp((n+1)\omega(x))\exp\left(m\varphi^*_\omega\left(\frac{|\alpha|}{m}\right)\right).
		\end{align*}
		Therefore, we have for every $\epsilon>0$ that 
		\begin{align*}
		r_{m,n+1}(\phi_\epsilon-f)&= \underset{\alpha \in \mathbb{N}^N_0}{\sup}\,\underset{x \in \mathbb{R}^N}{\sup}\, \exp\left(-m\varphi^*_\omega\left(\frac{|\alpha|}{m}\right)\right)\exp(-(n+1)\omega(x)) |\partial^\alpha \phi_\epsilon(x)-\partial^\alpha f(x)|\\& \leq  \underset{\alpha \in \mathbb{N}^N_0}{\sup}\,\underset{x \in \mathbb{R}^N}{\sup}\, \exp\left(-m\varphi^*_\omega\left(\frac{|\alpha|}{m}\right)\right)\exp(-(n+1)\omega(x))|\partial^\alpha f(x)(\phi(\epsilon x)-1)|\\
		&  +\epsilon C r_{M,n+1}(f)r_{M,0}(\phi).
		\end{align*}
		Since $f\in \bigcap_{m=1}^{\infty} \mathcal{O}^m_{n,\omega}(\mathbb{R}^N)$ and $\phi \in \mathcal{D}_\omega(\mathbb{R}^N)$,  $r_{M,n+1}(f)r_{M,0}(\phi)<\infty$ and so, we have $	\epsilon C r_{M,n+1}(f)r_{M,0}(\phi) \to 0$  as $\epsilon \to 0^+$. In order  to conclude the proof, it then remains  to prove that
		\begin{equation*}
		\underset{\alpha \in \mathbb{N}^N_0}{\sup}\,\underset{x \in \mathbb{R}^N}{\sup}\, \exp\left(-m\varphi^*_\omega\left(\frac{|\alpha|}{m}\right)\right)\exp(-(n+1)\omega(x)) |\partial^\alpha f(x)(\phi(\epsilon x)-1)|\to 0,\  \text{as}\; \epsilon \to 0^+.
		\end{equation*}
		But,  $\phi(\epsilon x)-1=0$ whenever $|x|\leq \frac{1}{\epsilon}$. Accordingly,  we have
		\begin{align*}
		&\underset{\alpha \in \mathbb{N}^N_0}{\sup}\,\underset{x \in \mathbb{R}^N}{\sup}\, \exp\left(-m\varphi^*_\omega\left(\frac{|\alpha|}{m}\right)\right)\exp(-(n+1)\omega(x)) |\partial^\alpha f(x)(\phi(\epsilon x)-1)|=\\&=\underset{\alpha \in \mathbb{N}^N_0}{\sup}\,\underset{|x|> \frac{1}{\epsilon} }{\sup}\, \exp\left(-m\varphi^*_\omega\left(\frac{|\alpha|}{m}\right)\right)\exp(-(n+1)\omega(x)) |\partial^\alpha f(x)(\phi(\epsilon x)-1)|\\& \leq r_{m,n}(f)\,\underset{|x|> \frac{1}{\epsilon} }{\sup}\, \exp(-\omega(x)).
		\end{align*}
		Since ${\sup_{|x|>\frac{1}{\epsilon}}}\, \exp(-\omega(x))\; \to 0$ as $\epsilon \to 0^+$, the claim is proved.
		
		From the arbitarity of $m\in\mathbb{N}$, we can conclude that $\phi_\epsilon \to f$ in $\bigcap_{m=1}^{\infty} \mathcal{O}^m_{n+1,\omega}(\mathbb{R}^N)$ as $\epsilon \to 0^+$ and hence in $\mathcal{O}_{C,\omega}(\mathbb{R}^N)$, taking into account that $\mathcal{O}_{C,\omega}(\mathbb{R}^N)$ is the inductive limit of the Fr\'echet spaces $\left\{\bigcap_{m=1}^{\infty}\mathcal{O}^{m}_{n,\omega}(\mathbb{R}^N)\right\}_{n\in\mathbb{N}}$.
	
	We now prove that $\mathcal{S}_{\omega}(\mathbb{R}^N)$ is a dense subspace of  $\mathcal{O}_{M,\omega}(\mathbb{R}^N)$, thereby obtaing by Theorem \ref{T.incl} that $\cO_{C,\omega}(\R^N)$ is also a dense subspace of $\cO_{M,\omega}(\R^N)$.  Since $\mathcal{D}_{\omega}(\mathbb{R}^N)\su \cS_\omega(\R^N)$, it suffices  to show only that $\mathcal{D}_{\omega}(\mathbb{R}^N)$ is a dense subspace of  $\mathcal{O}_{M,\omega}(\mathbb{R}^N)$. So, fix $f\in \mathcal{O}_{M,\omega}(\mathbb{R}^N)$ and  $\phi\in \mathcal{D}_{\omega}(\mathbb{R}^N)$ such that $\phi \equiv 1$ on $\overline{B_1(0)}$ and $0\leq \phi\leq 1$. Then for every $\epsilon>0$ the function $\phi_\epsilon (x):=\phi(\epsilon x)f(x)$, for $x\in \R^N$, belongs to  $ \mathcal{D}_{\omega}(\mathbb{R}^N)$ because $\cO_{M,\omega}(\R^N)\su \cE_\omega(\R^N)$. To show that $\varphi_\epsilon \to f$ in $\mathcal{O}_{M,\omega}(\mathbb{R}^N)$, we proceed as  follows.
		
		Fix $m\in \mathbb{N}$ and let $M\in \mathbb{N}$ such that $M\geq mL$, where $L\geq 1$ is the constant appearing in formula $(\ref{l})$. Since $f\in \mathcal{O}_{M,\omega}(\mathbb{R}^N)$, there exist $n(m), n(M)\in \mathbb{N}$ such that $f\in \mathcal{O}^m_{n(m),\omega}(\mathbb{R}^N)$ and $f\in \mathcal{O}^M_{n(M),\omega}(\mathbb{R}^N)$. Clearly, this implies that $f\in \mathcal{O}^m_{h,\omega}(\mathbb{R}^N)\cap \mathcal{O}^M_{h,\omega}(\mathbb{R}^N)$ for every $h\geq n:=\max\{n(M),n(m) \}$. Therefore, we can proceed as above to  show  that for every $\epsilon>0$ we have 
		\begin{align*}
		r_{m,n+1}(\phi_\epsilon-f) &\leq  \underset{\alpha \in \mathbb{N}^N_0}{\sup}\,\underset{x \in \mathbb{R}^N}{\sup}\, \exp\left(-m\varphi^*_\omega\left(\frac{|\alpha|}{m}\right)\right)\exp(-(n+1)\omega(x))\times \\
		&\quad \times |\partial^\alpha f(x)(\phi(\epsilon x)-1)|+\epsilon C r_{M,n+1}(f)r_{M,0}(\phi).
		\end{align*}
		Since $f\in \mathcal{O}^M_{n+1,\omega}(\mathbb{R}^N)$ and $\phi \in \mathcal{D}_\omega(\mathbb{R}^N)$, we have $r_{M,n+1}(f)r_{M,0}(\phi)<\infty$  and hence,  $\epsilon Cr_{M,n+1}(f)r_{M,0}(\phi)\to 0$ as $\epsilon\to 0^+$. In order to conclude the proof, it then remains to prove that
		\begin{equation*}
		\underset{\alpha \in \mathbb{N}^N_0}{\sup}\,\underset{x \in \mathbb{R}^N}{\sup}\, \exp\left(-m\varphi^*_\omega\left(\frac{|\alpha|}{m}\right)\right)\exp(-(n+1)\omega(x)) |\partial^\alpha f(x)(\phi(\epsilon x)-1)|\to 0,\ \text{as}\; \epsilon \to 0^+.
		\end{equation*}
		As before, we have  $\phi(\epsilon x)-1=0$ whenever $|x|\leq \frac{1}{\epsilon}$. Therefore,  for  every $\epsilon >0$ we have 
		\begin{align*}
		&\underset{\alpha \in \mathbb{N}^N_0}{\sup}\,\underset{x \in \mathbb{R}^N}{\sup}\, \exp\left(-m\varphi^*_\omega\left(\frac{|\alpha|}{m}\right)\right)\exp(-(n+1)\omega(x)) |\partial^\alpha f(x)(\phi(\epsilon x)-1)| \leq\\
		& \leq r_{m,n}(f)\,\underset{|x|> \frac{1}{\epsilon} }{\sup}\, \exp(-\omega(x)),
		\end{align*}
		where ${\sup}_{|x|> \frac{1}{\epsilon}}\, \exp(-\omega(x)) \to 0$ as $\epsilon \to 0^+$ and
		 $r_{m,n}(f)<\infty$. This means that $\phi_\epsilon\to f$ in the Banach space $\mathcal{O}^m_{n+1,\omega}(\mathbb{R}^N)$ as $\epsilon \to 0^+$, and hence in the (LB)-space $\bigcup_{h=1}^{\infty}\mathcal{O}^m_{h,\omega}(\mathbb{R}^N)$.  Since $m\in\N$ is arbitrary and $\mathcal{O}_{M,\omega}(\mathbb{R}^N)$ is the projective limit of the (LB)-spaces $\left\{\bigcup_{h=1}^{\infty}\mathcal{O}^m_{h,\omega}(\mathbb{R}^N)\right\}_{m\in\mathbb{N}}$, we can conclude that $\phi_\epsilon\to f$ in the space $\mathcal{O}_{M,\omega}(\mathbb{R}^N)$ as $\epsilon \to 0^+$. This completes the proof.
	\end{proof}

\section{$\cO_{M,\omega}(\R^N)$ is the space of multipliers of the spaces $\cS_\omega(\R^N)$ and $\cS'_\omega(\R^N)$ }

The main aim of this section is to prove that $\cO_{M,\omega}(\R^N)$ is the space of multipliers of both the spaces $\cS_\omega(\R^N)$ and  $\cS'_\omega(\R^N)$. In order to do this, we first show some preliminary facts.

\begin{lem}\label{L.Multip}
	Let $\omega$ be a non-quasianalytic weight function and $f\in C^\infty(\mathbb{R}^N)$. If $fg \in \mathcal{S}_{\omega}(\mathbb{R}^N)$ for every $g\in\mathcal{S}_{\omega}(\mathbb{R}^N)$, then  $f\in \mathcal{E}_{\omega}(\mathbb{R}^N)$.
\end{lem}
\begin{proof}
	Fix  a compact subset $K$ of $\mathbb{R}^N$ and  $g\in \mathcal{D}_\omega(\mathbb{R}^N)\su\mathcal{S}_\omega(\mathbb{R}^N) $ such that $g\equiv 1$ on $K$ and $0\leq g\leq 1$. Then $fg\in \mathcal{S}_{\omega}(\mathbb{R}^N)$ and so, for every $m\in \mathbb{N}$ we have
	\begin{align*}
	q_{m,1}(fg)=\underset{\alpha \in \mathbb{N}^N_0}{\sup}\,\underset{x \in \mathbb{R}^N}{\sup}\,|\partial^\alpha (fg)(x)| \exp\left(-m\varphi^*_\omega\left(\frac{|\alpha|}{m}\right)\right)\exp(\omega(x))  <\infty.
	\end{align*}
	 Since $g\equiv 1$ on $K$, it follows that
	\begin{align*}
	p_{K,m}(f)&=\underset{\alpha \in \mathbb{N}^N_0}{\sup}\, \underset{x \in K}{\sup}\, |\partial^\alpha f(x)|\exp\left(-m\varphi^*_\omega\left(\frac{|\alpha|}{m}\right)\right)\\& \leq \underset{\alpha \in \mathbb{N}^N_0}{\sup}\,\underset{x \in \mathbb{R}^N}{\sup}\,|\partial^\alpha (fg)(x)| \exp\left(-m\varphi^*_\omega\left(\frac{|\alpha|}{m}\right)\right)\exp(\omega(x)) \\ & = q_{m,1}(fg)<\infty.
	\end{align*}
	Since $K$ is arbitrary, we can conclude that $f\in \mathcal{E}_{\omega}(\mathbb{R}^N)$. 
\end{proof}

\begin{lem}\label{se}
	Let $\omega$ be a non-quasianalytic weight function. Let $\rho \in \mathcal{D}_{\omega}(\mathbb{R}^N)$ so that supp\,$\rho\subseteq B_1(0)$  and $\{x_j\}_{j\in\N}\subset \mathbb{R}^N$ be a sequence such that    $|x_j|\geq |x_{j-1}|+2$ for every $j\geq 2$ and $|x_1|>1$. If we set 
	\begin{equation}\label{eq.function}
	g(x):=\sum_{j\in\mathbb{N}} {\rho(x-x_j)}\exp(-j\omega(x_j)), \quad x\in\R^N,
	\end{equation}
	then $g \in \mathcal{S}_{\omega}(\mathbb{R}^N)$.
\end{lem}
\begin{proof} Since the functions of the sequence $\{\rho(\cdot-x_j)\}_{j\in\N}$ have disjoint supports, the series on the right of  \eqref{eq.function} converges in $C^\infty(\R^N)$ and so, the function $g$ on the left of \eqref{eq.function} belongs to  $C^\infty(\mathbb{R}^N)$. On the other hand, the property $(\gamma)$ of $\omega$ implies that $\exp(\omega(x))\geq \exp(a)(1+|x|)^b$ for every $x\in\mathbb{R}^N$. Therefore,  $g\in \mathcal{S}(\mathbb{R}^N)$, see \cite[Proposition 5, Chap. 4 \S 11]{H}. In order to conclude the proof, we proceed as follows.
	
	Fix $\lambda, \mu>0$ and $x\in\R^N$. Then either $x\in B_1(x_l)$ for some $j\in\N$ or $x\not\in B_1(x_l)$ for all $l\in\N$. Accordingly, we have for every $\alpha\in\N_0^N$ that either 
	\[
	\partial^\alpha g(x)=\exp(-j\omega(x_j))\partial^\alpha \rho(x-x_j), 
	\]
	or $\partial^\alpha g(x)=0$. Since $\rho\in \cD_\omega(\R^N)$, it follows for every $\alpha\in\N_0^N$ that  
	\begin{align}\label{eq.nuova}
&\exp(\mu\omega(x))	|\partial^\alpha g(x)|\leq\nonumber\\
&\leq \exp(\mu\omega(x)) \exp(-j\omega(x_j))\exp\left(-\mu K\omega(x-x_j)+\lambda\varphi^*_\omega\left(\frac{|\alpha|}{\lambda}\right)\right)q_{\lambda, K\mu}(\rho),
	\end{align}
	where $K$ is the constant appearing  in the property $(\alpha)$ of $\omega$. But, we have
	\[
	\omega(x)=\omega((x-x_j)+x_j)\leq K(\omega(x-x_j)+\omega(x_j)+1).
	\]
So, by \eqref{eq.nuova}	we obtain for every $\alpha\in\N_0^N$ that 
\begin{align*}
\exp(\mu\omega(x))	|\partial^\alpha g(x)|&\leq e^K\exp(K\mu\omega (x-x_j)+K\mu\omega(x_j))\times \\
&\times \exp(-j\omega(x_j))\exp\left(-\mu K\omega(x-x_j)+\lambda\varphi^*_\omega\left(\frac{|\alpha|}{\lambda}\right)\right)q_{\lambda, K\mu}(\rho)\\
&=e^K\exp((K\mu-j)\omega(x_j))\exp\left(\lambda\varphi^*_\omega\left(\frac{|\alpha|}{\lambda}\right)\right)q_{\lambda, K\mu}(\rho)
\end{align*}
and hence,
\[
\exp\left(-\lambda\varphi^*_\omega\left(\frac{|\alpha|}{\lambda}\right)\right)\exp(\mu\omega(x))	|\partial^\alpha g(x)|\leq e^K\exp((K\mu-j)\omega(x_j))q_{\lambda, K\mu}(\rho).
\]
Since $\exp((K\mu-j)\omega(x_j))\leq 1$ whenever $j\geq K\mu$ and $x$ is arbitrary, it follows that 
\[
q_{\lambda,\mu}(g)\leq e^K\max_{j< K\mu}\exp((K\mu-j)\omega(x_j))q_{\lambda, K\mu}(\rho)<\infty.
\]
But, $\lambda$ and $\mu$ are also arbitrary. So,  we can conclude that $g\in \cS_\omega(\R^N)$.
\end{proof}

\begin{rem}\label{R.se}  	Let $\omega$ be a non-quasianalytic weight function. Let $\rho \in \mathcal{D}_{\omega}(\mathbb{R}^N)$ so that supp\,$\rho\subseteq B_1(0)$  and $\{x_j\}_{j\in\N}\subset \mathbb{R}^N$ be a sequence such that    $|x_j|\geq |x_{j-1}|+2$ for every $j\geq 2$ and $|x_1|>1$. If $h\in \cS_\omega(\R^N)$ and 
	$g(x):=\sum_{j\in\mathbb{N}}h(x_j) \rho(x-x_j)$ for each  $x\in\R^N$,
	then $g \in \mathcal{S}_{\omega}(\mathbb{R}^N)$. The proof follows by argumenting as in the proof of Lemma \ref{se}, after having observed that $\sup_{x\in\R^N}|h(x)|\exp(j\omega(x))<\infty$ for each $j\in\N$.
	\end{rem}

We can now state and prove that $\cO_{M,\omega}(\R^N)$ is the space of multipliers of $\cS_\omega(\R^N)$. 

\begin{thm}\label{T.Multiplier} Let $\omega$ be a non-quasianalytic weight function and $f\in C^\infty(\R^N)$. Then the following properties are equivalent.
	\begin{enumerate}
		\item $f\in \cO_{M,\omega}(\R^N)$.
		\item  For every $g\in \cS_\omega(\R^N)$ we have  $fg\in \cS_{\omega}(\R^N)$.
	\end{enumerate}
Moreover, if $f\in\cO_{M,\omega}(\R^N)$, then the linear operator $M_f: \mathcal{S}_{\omega}(\mathbb{R}^N)\to\mathcal{S}_{\omega}(\mathbb{R}^N)$ defined by $M_f(g):=fg$, for  $g\in\mathcal{S}_{\omega}(\mathbb{R}^N)$, is   continuous.
	\end{thm}

\begin{proof} (1)$\Rightarrow$(2). 	Fix $g\in\mathcal{S}_{\omega}(\mathbb{R}^N)$. Then $fg \in C^\infty(\mathbb{R}^N)$. Moreover,  for every $\alpha \in \mathbb{N}^N_0$ and $x\in \mathbb{R}^N$ we have
	\begin{eqnarray}\label{ccc}
	|\partial^\alpha (fg)(x)|\leq \sum_{\gamma\leq\alpha}\binom{\alpha}{\gamma} |\partial^\gamma f(x)||\partial^{\alpha-\gamma} g(x)|.
	\end{eqnarray}
	Fixed $m\in \mathbb{N}$,  let $M\in \mathbb{N}$ so that $M\geq Lm$, where $L\geq 1$ is the constant appearing in formula \eqref{l}. Since $f\in \mathcal{O}_{M,\omega}(\mathbb{R}^N)$, there exist $n\in\mathbb{N}$ and $C>0$ such that for every $\gamma \in \mathbb{N}^N_0$ and $x\in \mathbb{R}^N$ we have
	\begin{equation}\label{bbb}
	|\partial^\gamma f(x)|\leq C\exp \left(n\omega(x)+M\varphi_\omega^*\left(\frac{|\gamma|}{M}\right)\right).		
	\end{equation}
	 On the other hand,   $g\in\mathcal{S}_{\omega}(\mathbb{R}^N)$. So,  setting $M':=\max\{M, n+m\}$ and recalling that  $\varphi^*_\omega(t)/t$ is an increasing function in $(0,\infty)$, we have for  every $\delta \in \mathbb{N}^N_0$ and $x\in \mathbb{R}^N$ that 
	\begin{align}\label{aaa}
	|\partial^\delta g(x)|&\leq q_{M',M'}(g) \exp \left(M'\varphi^*_\omega\left(\frac{|\delta|}{M'}\right)-M'\omega(x)\right)\nonumber\\
	&\leq q_{M',M'}(g) \exp \left(M\varphi^*_\omega\left(\frac{|\delta|}{M}\right)-M'\omega(x)\right).
	\end{align}
	By combining $(\ref{ccc})$, $(\ref{bbb})$ and $(\ref{aaa})$, we obtain for every $\alpha\in \N_0^N$ and $x\in \R^N$ that
	\begin{align*}
	&\exp(m\omega(x))|\partial^\alpha (fg)(x)|\leq\\&\leq \sum_{\gamma\leq\alpha}\binom{\alpha}{\gamma} C\exp((n+m)\omega(x))\exp \left(M\varphi^*_\omega\left(\frac{|\gamma|}{M}\right)\right)|\partial^{\alpha-\gamma} g(x)|\\& \leq C \sum_{\gamma\leq\alpha}\binom{\alpha}{\gamma}\exp((n+m)\omega(x))\exp \left(M\varphi^*_\omega\left(\frac{|\gamma|}{M}\right)\right)q_{M',M'}(g) \times\\
	&\times \exp \left(M\varphi^*_\omega\left(\frac{|\alpha-\gamma|}{M}\right)-M'\omega(x)\right) \\& \leq C q_{M',M'}(g)\exp((n+m-M')\omega(x))\sum_{\gamma\leq\alpha}\binom{\alpha}{\gamma}\exp \left(M\varphi^*_\omega\left(\frac{|\gamma|}{M}\right)\right)\times \\
	&\times \exp \left(M\varphi^*_\omega\left(\frac{|\alpha-\gamma|}{M}\right)\right) .
	\end{align*}
Using inequality (\ref{secondprop}) and taking in mind that $\sup_{x\in\R^N}\exp((n+m-M')\omega(x))<\infty$, it follows for every $\alpha\in\N_0^N$ and $x\in\R^N$ that 
	\begin{equation}\label{eq.finale}
		\exp(m\omega(x))|\partial^\alpha (fg)(x)|\leq C q_{M',M'}(g)2^{|\alpha|}\exp \left(M\varphi^*_\omega\left(\frac{|\alpha|}{M}\right)\right).
	\end{equation}
	Since $M\geq mL$, applying formula (\ref{firstprop}) we obtain via \eqref{eq.finale} that for every $\alpha\in\N_0^N$ and $x\in\R^N$ 
	\begin{equation*}
	\exp(m\omega(x))|\partial^\alpha (fg)(x)|\leq C q_{M',M'}(g)\exp \left(m\varphi^*_\omega\left(\frac{|\alpha|}{m}\right)\right) 
	\end{equation*}
	and so
	\begin{equation}\label{eq.ConM}
	q_{m,m}(fg)\leq C \exp(mL) q_{M',M'}(g)<\infty.
	\end{equation}
	The thesis follows from the arbitrarity of $m\in \mathbb{N}$.

	
	(2)$\Rightarrow$(1). We first observe that by Lemma \ref{L.Multip} we have $f\in \mathcal{E}_{\omega}(\mathbb{R}^N)$. We now suppose that $f\notin \mathcal{O}_{M,\omega}(\mathbb{R}^N)$. Then by Proposition \ref{P.NcharO}(1) there exists $m\in \mathbb{N}$ such that for each $C,R>0$ and $n\in \mathbb{N}$ there exist $x\in \mathbb{R}^N$ with $|x|>R$ and $\alpha\in \mathbb{N}^N_0$ such that
	\begin{equation*}
	|\partial^\alpha f(x)|\geq C\exp\left(n\omega(x)+m\varphi^*_\omega\left(\frac{|\alpha|}{m}\right)\right).
	\end{equation*}
	So, we can choose two sequences $\{\alpha_j\}_{j\in\N} \subset \mathbb{N}^N_0$ and $\{x_j\}_{j\in\N} \subset \mathbb{R}^N$ with $|x_{j+1}|>|x_j|+2$ for all $j\in\N$ and $|x_1|>1$ such that
	\begin{equation}\label{eq.stima}
	|\partial^{\alpha_j} f(x_j)|\geq \exp\left(j\omega(x_j)+m\varphi^*_\omega\left(\frac{|\alpha_j|}{m}\right)\right).
	\end{equation}
	Let $\rho \in \mathcal{D}_\omega(\mathbb{R}^N)\subset\mathcal{S}_\omega(\mathbb{R}^N)$ such that ${\rm supp}\, \rho\su B_1(0)$ and  $\rho \equiv 1$ on $B_{r}(0)$ for some $0<r<1$. For every $x\in \mathbb{R}^N$ let
	\begin{equation*}
	g(x):=\sum_{j\in\mathbb{N}} {\rho(x-x_j)}\exp(-j\omega(x_j)).
	\end{equation*}
	Then by Lemma \ref{se} we have  $g \in \mathcal{S}_{\omega}(\mathbb{R}^N)$ and so,  $q_{m,m}(fg)<\infty$. In particular, we have that 
	\begin{equation}\label{eq.stima1}
 \underset{j\in\mathbb{N}}{\sup}\, |\partial^{\alpha_j} (fg)(x_j)|\exp\left(m\omega(x_j)-m\varphi^*_\omega\left(\frac{|\alpha_j|}{m}\right)\right)\leq 	q_{m,m}(fg)<\infty.
	\end{equation}
	On the other hand, for every $j\in\N$ we have $g\equiv {\exp(-j\omega(x_j))}$ in $\overline{B}_{r}(x_j)$. Accordingly,  for every $j\in\N$, $x\in \overline{B}_{r}(x_j)$ and $\alpha\in\N_0^N$ we have
	\begin{equation*}
	\partial^{\alpha}(fg)(x)={\exp(-j\omega(x_j))} \partial^\alpha f(x).
	\end{equation*}
Therefore, by \eqref{eq.stima} and \eqref{eq.stima1} it follows that 
	\begin{align*}
	q_{m,m}(fg)&\geq \underset{j\in\mathbb{N}}{\sup}\, |\partial^{\alpha_j} f(x_j)|\exp\left(m\omega(x_j)-m\varphi^*_\omega\left(\frac{|\alpha_j|}{m}\right)\right){\exp(-j\omega(x_j))}\\
	&\geq \sup_{j\in\N}\exp(m\omega(x_j)),
	\end{align*}
thereby implying that $\sup_{j\in\N}\exp(m\omega(x_j))<\infty$. But, $\lim_{j\to\infty}\exp(m\omega(x_j)) =+\infty$ and hence, $\sup_{j\in\N}\exp(m\omega(x_j))=\infty$.	This is a contradiction.

Fix any $f\in \cO_{M,\omega}(\R^N)$. Then the operator $M_f\colon \cS_\omega(\R^N)\to \cS_\omega(\R^N)$ is well defined by the proof above. Actually, $M_f$ is also continuous as follows directly from \eqref{eq.ConM}.
	\end{proof}

\begin{rem}\label{R.Multiplier}
(a)	Let $f\in \cS_\omega(\R^N)$. Then $fg\in \cS_\omega(\R^N)$ for every $g\in \cS_\omega(\R^N)$. So, by Theorem \ref{T.Multiplier} it follows that $f\in \cO_{M,\omega}(\R^N)$.
	
(b) Let $f\in \mathcal{E}_{\omega}(\mathbb{R}^N)$. Then $fg \in  \mathcal{S}_{\omega}(\mathbb{R}^N)$ for every $g\in \mathcal{S}_{\omega}(\mathbb{R}^N)$ if and only if $(\partial^\alpha f)g \in  \mathcal{S}_{\omega}(\mathbb{R}^N)$ for every $g\in \mathcal{S}_{\omega}(\mathbb{R}^N)$ and $\alpha \in \mathbb{N}^N_0$.

Indeed, the condition on the right clearly implies the condition on the left. Conversely, the assumption $fg \in \mathcal{S}_\omega(\mathbb{R}^N)$ for every $g \in \mathcal{S}_\omega(\mathbb{R}^N)$ implies   for every $j=1,\dots,N$ that
\begin{equation*}
(\partial_j f) g \,= \partial_j (fg) - f(\partial_j g) \in \mathcal{S}_\omega(\mathbb{R}^N).
\end{equation*}
So, proceeding  by induction the result follows.

Consequently, by Theorem \ref{T.Multiplier} we can conclude that  for  fixed $f\in \mathcal{O}_{M,\omega}(\mathbb{R}^N)$ and $\alpha \in \mathbb{N}^N_0$ the function $\partial^\alpha f\in \cO_{M,\omega}(\R^N)$ too, and the linear  operator $M_{\partial^\alpha f}: \mathcal{S}_{\omega}(\mathbb{R}^N)\to\mathcal{S}_{\omega}(\mathbb{R}^N)$ is continuous.
\end{rem}

Finally, we show that $\cO_{M,\omega}(\mathbb{R}^N)$ is also the space of multipliers of $\mathcal{S}'_\omega(\mathbb{R}^N)$.
\begin{thm}\label{T.Multiplierdual}
	Let $\omega$ be a non-quasianalytic weight function and $f\in \cE_\omega(\R^N)$. Then the following properties are equivalent. 
	\begin{enumerate} 
		\item $f\in \cO_{M,\omega}(\mathbb{R}^N)$. 
		 \item For every $T\in \mathcal{S}'_\omega(\mathbb{R}^N)$ we have $fT \in \mathcal{S}'_\omega(\mathbb{R}^N)$.  
		 \end{enumerate}
	  Moreover, if $f\in \cO_{M,\omega}(\mathbb{R}^N)$, then the linear operator  $\cM_f\colon \mathcal{S}'_\omega(\mathbb{R}^N)\to  \mathcal{S}'_\omega(\mathbb{R}^N)$ defined by $\cM_f(T):=fT$, for $T\in \mathcal{S}'_\omega(\mathbb{R}^N)$, is continuous.
\end{thm}
\begin{proof} (1)$\Rightarrow$(2). Fix $T\in \cS_\omega'(\R^N)$. 
	 By Theorem \ref{T.Multiplier} the linear operator  $M_f\colon \mathcal{S}_\omega(\mathbb{R}^N)\to \mathcal{S}_\omega(\mathbb{R}^N)$ given by $M_f(g)=fg$, for $g\in\cS_\omega(\R^N)$,  is continuous and so,  the linear functional $T\circ M_f\colon \cS_\omega(\R^N)\to \C$ is also continuous. Accordingly, $T\circ M_f\in \cS'_\omega(\R^N)$. But, for every $g\in \cS_\omega(\R^N)$ we have
	 \[
	 (T\circ M_f)(g)=\langle T, M_f (g)\rangle=\langle T, f g \rangle=\langle fT,  g \rangle.
	 \]
	 Therefore, $fT=T\circ M_f\in \cS'_\omega(\R^N)$. This completes the proof.

	  (2)$\Rightarrow$(1).  
Suppose for every $T\in \mathcal{S}'_\omega(\mathbb{R}^N)$ that  $fT \in \mathcal{S}'_\omega(\mathbb{R}^N)$. Then for a fixed $g\in \mathcal{S}_\omega(\mathbb{R}^N)$ the linear functional $L\colon  \cS'_\omega(\R^N)\to  \C$ defined by $L(T):=\langle fT,g\rangle$, for   $T\in \cS'_\omega(\R^N)$, is continuous and so, $L\in \cS''_\omega(\R^N)$. Since $\cS_\omega(\R^N)$ is reflexive,  there exists $h\in \mathcal{S}_\omega(\mathbb{R}^N)$ such that $L(T)=\langle T,h\rangle$ for every $T\in \mathcal{S}'_\omega(\mathbb{R}^N)$, i.e., $\langle fT,g\rangle =\langle T,h\rangle$ for every $T\in \mathcal{S}'_\omega(\mathbb{R}^N)$. Accordingly,  for every $T\in \mathcal{D}_\omega(\mathbb{R}^N)\su \cS'_\omega(\R^N)$ we have $\langle fT,g\rangle =\langle T,h\rangle$. Since for every $T\in \cD_\omega(\R^N)$ we have $\langle fT,g\rangle =\langle T,fg\rangle$ and $\cD_\omega(\R^N)$  is a dense subspace of $\cS_\omega(\R^N)$, it follows that  $fg=h\in \mathcal{S}_\omega(\mathbb{R}^N)$. So, as $g\in \cS_\omega(\R^N)$ is arbitrary, by Theorem \ref{T.Multiplier} we can conclude that  $f\in \cO_{M,\omega}(\mathbb{R}^N)$.

Fix any $f\in \cO_{M,\omega}(\R^N)$. 	 Then  for every $g\in \mathcal{S}_\omega(\mathbb{R}^N)$ and $T\in \cS'_\omega(\R^N)$ we have  
	\begin{equation*}
	\langle fT, g\rangle =\langle T, fg\rangle.
	\end{equation*}
This   means that the linear operator $\cM_f\colon \cS'_\omega(\R^N)\to \cS'_\omega(\R^N)$ is the transpose of the continuous linear operator $M_f\colon \cS_\omega(\R^N)\to \cS_\omega(\R^N)$. Therefore, the linear operator  $\cM_f\colon \cS'_\omega(\R^N)\to \cS'_\omega(\R^N)$ is necessarily continuous.
	\end{proof}

\section{Other topologies on $\cO_{M,\omega}(\R^N)$}

In this section we show that the space $\cO_{M,\omega}(\R^N)$ can be naturally endowed with other lc-topologies. We also compare these lc-topologies with each other and with the projective lc-topology defined on $\cO_{M,\omega}(\R^N)$ by the  spectrum 
$\{\cup_{n=1}^\infty \cO^m_{n,\omega}(\R^N)\}_{m\in\N}$. 

We begin by giving another useful chartacterization of the space $\cO_{M,\omega}(\R^N)$.

\begin{thm}\label{T.top} Let $\omega$ be a non-quasianalytic weight function and $f\in C^\infty(\R^N)$. Then the following properties are equivalent.
	\begin{enumerate}
		\item $f\in \cO_{M,\omega}(\R^N)$.
		\item For every $g\in \cS_\omega(\R^N)$ and $m\in\N$ we have
		\begin{equation}
		q_{m, g}(f):=\sup_{\alpha\in\N_0^N}\sup_{x\in \R^N}|g(x)||\partial^\alpha f(x)|\exp\left(-m\varphi^*_\omega\left(\frac{|\alpha|}{m}\right)\right)<\infty.
		\end{equation}
	\end{enumerate}
	\end{thm}

\begin{proof} (1)$\Rightarrow$(2). Since $f\in \cO_{M,\omega}(\R^N)$, we have  for every $m\in\N$ that there exist $C>0$ and $n\in\N$ such that for every $\alpha\in\N_0^N$ and $x\in \R^N$ the following inequality is satisfied
	\[
	|\partial^\alpha f(x)|\leq C\exp\left(n\omega(x)+m\varphi^*_\omega\left(\frac{|\alpha|}{m}\right)\right).
	\]
	Therefore, for any fixed $g\in \cS_\omega(\R^N)$, it follows for every $m\in\N$, $\alpha\in\N_0^N$ and $x\in\R^N$ that 
	\[
	|g(x)|	|\partial^\alpha f(x)|\exp\left(-m\varphi^*_\omega\left(\frac{|\alpha|}{m}\right)\right)\leq C\exp(n\omega(x))|g(x)|,
	\]
	with $n\in\N$ depending only of  $m$. This implies that 
	\[
	q_{m,g}(f)\leq C\sup_{x\in\R^N}\exp(n\omega(x))|g(x)|<\infty.
	\]
	Since $g\in \cS_{\omega}(\R^N)$ is arbitrary, the thesis follows.
	
	(2)$\Rightarrow$(1). We first show that   $f\in \cE_\omega(\R^N)$. So, fixed any compact subset $K$ of $\R^N$, let $g\in \cD_\omega(\R^N)$ such that $g\equiv 1$ on $K$. Then it follows  for every $m\in\N$ that 
	\begin{align}\label{eq.continuita}
	p_{K,m}(f)&=\sup_{x\in K}\sup_{\alpha\in \N_0^N}|\partial ^\alpha f(x)|\exp\left(-m\varphi^*_\omega\left(\frac{|\alpha|}{m}\right)\right)\nonumber\\
	&\leq \sup_{x\in K}\sup_{\alpha\in \N_0^N}|g(x)||\partial ^\alpha f(x)|\exp\left(-m\varphi^*_\omega\left(\frac{|\alpha|}{m}\right)\right)\nonumber\\
	&\leq q_{m,g}(f)<\infty.
	\end{align}
	Since $K$ is an arbitrary compact subset of $\R^N$, we can conclude that $f\in\cE_\omega(\R^N)$.
	
	Suppose that $f\not\in \cO_{M,\omega}(\R^N)$. Since $f\in \cE_\omega(\R^N)$, it follows from  Proposition \ref{P.NcharO}(1) that  there exists $m\in\N$ such that for each $C,R>0$ and $n\in\N$ there exist $x\in\R^N$ with $|x|>R$ and $\alpha\in\N_0^N$ such that
	\[
	|\partial^\alpha f(x)|\geq C \exp\left(n\omega(x)+m\varphi^*_\omega\left(\frac{|\alpha|}{m}\right)\right).
	\]
		So, we can choose two sequences $\{\alpha_j\}_{j\in\N} \subset \mathbb{N}^N_0$ and $\{x_j\}_{j\in\N} \subset \mathbb{R}^N$ with $|x_{j+1}|>|x_j|+2$ for all $j\in\N$ and $|x_1|>1$ such that
	\begin{equation}\label{eq.stimaL}
	|\partial^{\alpha_j} f(x_j)|\geq \exp\left(j\omega(x_j)+m\varphi^*_\omega\left(\frac{|\alpha_j|}{m}\right)\right).
	\end{equation}
Next, let $\rho \in \mathcal{D}_\omega(\mathbb{R}^N)\subset\mathcal{S}_\omega(\mathbb{R}^N)$ such that ${\rm supp}\, \rho\su B_1(0)$ and  $\rho \equiv 1$ on $B_{r}(0)$ for some $0<r<1$. For every $x\in \mathbb{R}^N$ let
	\begin{equation*}
	g(x):=\sum_{j\in\mathbb{N}} {\rho(x-x_j)}\exp(-j\omega(x_j)).
	\end{equation*}
	Then by Lemma \ref{se} we have $g\in \cS_\omega(\R^N)$ and hence, it necessarily holds that 
	$q_{m,g}(f)<\infty$. In particular, we have that 
	\begin{equation}\label{eq.stiam2L}
	\sup_{j\in\N}|g(x_j)||\partial^{\alpha_j} f(x_j)|\exp\left(m\omega(x_j)-m\varphi^*_\omega\left(\frac{|\alpha_j|}{m}\right)\right)\leq q_{m,g}(f)<\infty.
	\end{equation}
	But, $g(x_j)= \exp(-j\omega(x_j))$  for every $j\in\N$. So, it follows by \eqref{eq.stimaL} and \eqref{eq.stiam2L} that
	\begin{align*}
	q_{m,g}(f)&\geq 	\sup_{j\in\N}\exp(-j\omega(x_j)||\partial^{\alpha_j} f(x_j)|\exp\left(m\omega(x_j)-m\varphi^*_\omega\left(\frac{|\alpha_j|}{m}\right)\right)\\
	&\geq \sup_{j\in\N}\exp(m\omega(x_j)),
	\end{align*}
	thereby implying that  $\sup_{j\in\N}\exp(m\omega(x_j))<\infty$. This is a contradiction because $\lim_{j\to\infty}\exp(m\omega(x_j))=\infty$. Hence, $f$  necessarily belongs to $\cO_{M,\omega}(\R^N)$. 
\end{proof}

Theorem \ref{T.top} implies that  the set $\{q_{m,g}\}_{m\in\N, g\in \cS_\omega(\R^N)}$ forms a fundamental system of norms on $\cO_{M,\omega}(\R^N)$. Denote by $\tau$ the Hausdorff lc-topology on $\cO_{M,\omega}(\R^N)$ generated by $\{q_{m,g}\}_{m\in\N, g\in \cS_\omega(\R^N)}$.  

We now collect some properties of the Hausdorff lc-space $(\cO_{M,\omega}(\R^N),\tau)$.

\begin{thm}\label{T.Incl} Let $\omega$ be a non-quasianalytic weight function. Then the follwing properties are satisfied.
	\begin{enumerate}
		\item The inclusion 
		\begin{equation}\label{eq.Inc}
		(\cO_{M,\omega}(\R^N),\tau)\hookrightarrow \cE_{\omega}(\R^N)
		\end{equation}
		is continuous with dense range.
		\item $(\cO_{M,\omega}(\R^N),\tau)$ is a complete lc-space.
	\end{enumerate}
	\end{thm}

\begin{proof} (1) The continuity of the inclusion $(\cO_{M,\omega}(\R^N), \tau)\hookrightarrow \cE_\omega(\R^N)$ follows by repeating  the argument at the beginning of the proof of Theorem (\ref{T.top}) (2)$\Rightarrow$(1), i.e., of \eqref{eq.continuita}. On the other hand, the facts that  the inclusion $\cD_\omega(\R^N)\hookrightarrow \cE_\omega(\R^N)$ has dense range (see Remark \ref{R.PspaziB}(3)) and $\cD_\omega(\R^N)\su \cO_{M,\omega}(\R^N)$ (see Remark \ref{R.spaziD}(1)) clearly imply that  the inclusion $(\cO_{M,\omega}(\R^N), \tau)\hookrightarrow \cE_\omega(\R^N)$ has dense range too.

(2) Let $\{f_i\}_{i\in I}$ be a Cauchy net in $(\cO_{M,\omega}(\R^N),\tau)$. Since the inclusion  $(\cO_{M,\omega}(\R^N),\tau)\hookrightarrow C^\infty(\R^N)$ is continuous as it is easy to prove, it follows that $\{f_i\}_{i\in I}$ is also  a Cauchy net in $C^\infty(\R^N)$. But, $C^\infty(\R^N)$ is a Fr\'echet space and hence, a complete lc-space. So, there exists $f\in C^\infty(\R^N)$ such that $f_i\to f$ in $C^\infty(\R^N)$. We claim that  $f_i\to f$ in $(\cO_{M,\omega}(\R^N),\tau)$. To see this, we fix $g\in \cS_\omega(\R^N)$, $m\in\N$ and $\epsilon>0$. Since $\{f_i\}_{i\in I}$ is a Cauchy net in $(\cO_{M,\omega}(\R^N),\tau)$, there exists $i_0\in I$ such that for every $i,i'\geq i_0$ we have
\[
q_{m,g}(f_i-f_{i'})=\sup_{x\in\R^N}\sup_{\alpha\in\N_0^n}|g(x)||\partial^\alpha(f_i-f_{i'})(x)|\exp\left(-m\varphi^*_\omega\left(\frac{|\alpha|}{m}\right)\right)<\epsilon,
\]
i.e., for every $i,i'\geq i_0$, $\alpha\in\N_0^N$ and $x\in\R^N$ we have
\[
|g(x)||\partial^\alpha(f_i-f_{i'})(x)|<\epsilon \exp\left(m\varphi^*_\omega\left(\frac{|\alpha|}{m}\right)\right).
\]
Since  $f_i\to f$ in $C^\infty(\R^N)$ implies that $\partial^\alpha f_i\to \partial^\alpha f$ pointwise in $\R^N$ for every $\alpha\in\N_0^N$, by letting $i'$ to infty it follows for every $i\geq i_0$, $\alpha\in\N_0^N$ and $x\in\R^N$  that 
\[
|g(x)||\partial^\alpha(f_i-f)(x)|<\epsilon \exp\left(m\varphi^*_\omega\left(\frac{|\alpha|}{m}\right)\right)
\]
and so,
\begin{align*}
|g(x)||\partial^\alpha f(x)|&\leq |g(x)||\partial^\alpha(f-f_{i_0})(x)|+|g(x)||\partial^\alpha f_{i_0}(x)|\\
& <\epsilon \exp\left(m\varphi^*_\omega\left(\frac{|\alpha|}{m}\right)\right)+|g(x)||\partial^\alpha f_{i_0}(x)|.
\end{align*}
Accordingly, we have for every $i\geq i_0$ that 
\[
q_{m,g}(f_i-f)\leq \epsilon \quad {\rm and }\quad q_{m,g}(f)\leq \epsilon+ q_{m,g}(f_{i_0}).
\]
Since $g\in \cS_\omega(\R^N)$ and $m\in\N$ are arbitrary, this shows via Theorem \ref{T.top} that $f\in \cO_{M,\omega}(\R^N)$ and that $f_i\to f$ in $(\cO_{M,\omega}(\R^N),\tau)$. 
\end{proof}

We recall that the sequence $\{\cup_{n=1}^\infty \cO^m_{n,\omega}(\R^N)\}_{m\in\N}$ of (LB)-spaces forms a projective spectrum and $\cO_{M,\omega}(\R^N)=\cap_{m=1}^\infty\cup_{n=1}^\infty \cO^m_{n,\omega}(\R^N)$. We denote by $t$ the  projective topology  on the space  $\cO_{M,\omega}(\R^N)$ defined by  $\{\cup_{n=1}^\infty \cO^m_{n,\omega}(\R^N)\}_{m\in\N}$. The next aim is to compare the topology $\tau$ with the topology $t$. To this end, 
	we  introduce the following spaces.
	
	\begin{defn}\label{D.spaziOV} Let $\omega$ be a non-quasianalytic weight function. For $m\in\N$ we define the space
		\begin{equation}\label{eq.OV}
		\cO^m_{s_\omega}(\R^N):=\left\{f\in C^\infty(\R^N)\colon\forall g\in \cS_\omega(\R^N)\ \  q_{m,g}(f)<\infty\right\}
		\end{equation}
		and endow it with the lc-topology $\tau_m$ generated by the system of norms $\{q_{m,g}\}_{g\in \cS_\omega(\R^N)}$.
	 \end{defn}
 
 It is straightforward to verify that the following  topological equality holds
 \begin{equation}\label{eq.identita}
 (\cO_{M,\omega}(\R^N),\tau)=\cap_{m=1}^\infty (\cO^m_{s_\omega}(\R^N),\tau_m),
 \end{equation}
 when the space on the right hand side is endowed with the corresponding projective
 limit topology.
Moreover, the following results hold.
 
 \begin{prop}\label{P.propOV} Let $\omega$ be a non-quasianalytic weight function and $m\in\N$. Then the following properties are satisfied.
 	\begin{enumerate}
 			\item The inclusion $\cup_{n=1}^\infty\cO^m_{n,\omega}(\R^N)\hookrightarrow (\cO^m_{s_\omega}(\R^N),\tau_m)$ is well-defined and continuous.
 			\item $\cup_{n=1}^\infty\cO^m_{n,\omega}(\R^N)=\cO^m_{s_\omega}(\R^N)$ algebraically. Moreover, the spaces $\cup_{n=1}^\infty\cO^m_{n,\omega}(\R^N)$ and  $(\cO^m_{s_\omega}(\R^N),\tau_m)$ have the same bounded sets. 
 		\item $(\cO^m_{s_\omega}(\R^N),\tau_m)$ is a complete lc-space.
 	\end{enumerate}
 	\end{prop}
 
 \begin{proof} (1) Fix $g\in \cS_{\omega}(\R^N)$. Then for each $n\in\N$ there exists $c_n>0$ such that for every $x\in \R^N$ we have
 	\begin{equation}\label{eq.stt}
 	|g(x)|\leq c_n\exp (-n\omega(x)).
 	\end{equation}
 	This implies for every $n\in\N$ that the inclusion
 	\[
 	\cO^m_{n,\omega}(\R^N)\hookrightarrow \cO^m_{s_\omega}(\R^N)
 	\]
 	is well-defined and continuous. Indeed, for a fixed $n\in\N$,  we obtain  via \eqref{eq.stt} that
 	\begin{align*}
 	q_{m,g}(f)&=\sup_{x\in\R^N}\sup_{\alpha\in\N_0^N}|g(x)||\partial^\alpha f(x)|\exp\left(-m\varphi^*_\omega\left(\frac{|\alpha|}{m}\right)\right)\\
 	&\leq c_n \sup_{x\in\R^N}\sup_{\alpha\in\N_0^N}|\partial^\alpha f(x)|\exp\left(-n\omega(x)-m\varphi^*_\omega\left(\frac{|\alpha|}{m}\right)\right)=c_nr_{m,n}(f).
 	\end{align*} 
 	Since $g\in\cS_\omega(\R^N)$ is arbitrary,
 	the continuity of the inclusion  $\cO^m_{n,\omega}(\R^N)\hookrightarrow \cO^m_{s_\omega}(\R^N)$ follows. 
 	
 	Since $\cup_{n=1}^\infty\cO_{n,\omega}^m(\R^N)$ is an (LB)-space, we deduce that the inclusion $\cup_{n=1}^\infty\cO^m_{n,\omega}(\R^N)\hookrightarrow (\cO^m_{s_\omega}(\R^N),\tau_m)$ is continuous. 
 	
 (2)	By (1) above it suffices to show only  that every bounded subset of  $(\cO^m_{s_\omega}(\R^N),\tau_m)$ is also a bounded subset of  $\cup_{n=1}^\infty\cO_{n,\omega}^m(\R^N)$. To this end, we  fix a bounded subset $B$ of $(\cO^m_{s_\omega}(\R^N),\tau_m)$. We would show that there exists $n_0\in\N$ such that $\sup_{f\in B}r_{m,n_0}(f)<\infty$. If this is not the case, then $\sup_{f\in B}r_{m,n}(f)=\infty$ for every $n\in\N$.  To get a contradiction, we proceed as follows.
 	
 	For each $j\in\N$ let $K_j:=\ov{B}_j(0)\times\{\alpha\in\N_0^N\colon |\alpha|\leq j\}$. Then $\cup_{j\in\N}K_j=\R^N\times \N_0^N$.  On the other hand, taking in account that the function $g_0(x):=\exp(-|x|^2)$, for $x\in\R^N$, belongs to $\cS_\omega(\R^N)$ we have for every $f\in B$ and $n,j\in\N$ that 
 	\begin{align*}
	&\sup_{(x,\alpha)\in K_j}|\partial^\alpha f(x)|\exp\left(-n\omega(x)-m\varphi^*_\omega\left(\frac{|\alpha|}{m}\right)\right)\leq\\
	&\leq \sup_{(x,\alpha)\in K_j}\exp(|x|^2-n\omega(x))|\partial^\alpha f(x)|g_0(x)\exp\left(-m\varphi^*_\omega\left(\frac{|\alpha|}{m}\right)\right)\leq k_j q_{m,g_0}(f),
 	\end{align*}
 	where $k_j:=\sup_{(x,\alpha)\in K_j}\exp(|x|^2-n\omega(x))<\infty$, and hence
 	\[
 	\sup_{f\in B}\sup_{(x,\alpha)\in K_j}|\partial^\alpha f(x)|\exp\left(-n\omega(x)-m\varphi^*_\omega\left(\frac{|\alpha|}{m}\right)\right)<\infty.
 	\]
 Since $\sup_{f\in B} r_{m,n}(f)=\infty$, an argument by induction then yields that there exist a sequence $\{f_n\}_{n\in\N}\su B$ and  a strictly increasing sequence $\{j_n\}_{n\in\N}$ of positive integers such  that for every $n\in\N$ there exists some point $(x_n,\alpha_n)\in \stackrel{\circ}{K}_{j_{n+1}}\setminus K_{j_{n-1}}$ for which
 	\begin{equation}\label{eq.contr}
 	|\partial^{\alpha_n}f_n(x_n)|\exp\left(-n\omega(x_n)-m\varphi^*_\omega\left(\frac{|\alpha|}{m}\right)\right)>n.
 	\end{equation}
 	 Now, we choose   a function $\rho\in\cD_\omega(\R^N)$ such that ${\rm supp}\rho\su B_1(0)$ and $\rho(0)=1$.
 	  We define 
 	\begin{equation}\label{eq.nuovaf}
 	g(x):=\sum_{k=1}^\infty \exp(-k\omega(x))\rho(x-x_k), \quad x\in \R^N. 
 	\end{equation}
 	 Then by Lemma \ref{se} we have $g\in\cS_\omega(\R^N)$. In particular, $g(x_n)=\exp(-n\omega(x_n))$ for every $n\in\N$.
 	Hence, by \eqref{eq.contr} it follows for every $n\in\N$  that
 	\begin{align*}
 	q_{m,g}(f_n)&\geq g(x_n)|\partial^{\alpha_n} f_n(x_n)|\exp\left(-m\varphi^*_\omega
 	\left(\frac{|\alpha_n|}{m}\right)\right)\\
 	&\geq  \exp(-n\omega(x_n))|\partial^{\alpha_n} f_n(x_n)|\exp\left(-m\varphi^*_\omega
 	\left(\frac{|\alpha_n|}{m}\right)\right)>n.
 	\end{align*}
 	This shows that 
$	\sup_{f\in B}q_{m,g}(f)=\infty$,
 	 which is a contradiction as $B$ is a bounded subset of  $( \cO^m_{s_\omega}(\R^N),\tau)$ and hence $\sup_{f\in B}q_{m,g}(f)<\infty$.

 	 (3) follows as in the proof of Theorem \ref{T.Incl}(2).
 \end{proof}

Let $X=\ind_{n\rightarrow}X_n$  be an (LB)-space with canonical inclusions $j_n\colon X_n\to X$ for each $n\in\N$. Recall that $X$ is called \textit{regular} if every bounded subset of $X$ is contained and bounded in a step $X_m$ for some $m\in \N$. Every complete (LB)-space is regular, \cite[(5) p.225]{Ko}. Accordingly,  Proposition \ref{P.LBcomplete} implies that the space $\cup_{n=1}^\infty \cO^m_{n,\omega}(\R^N)$ is a regular  (LB)-space for each $m\in\N$. On the other hand, 
 Proposition \ref{P.propOV} yields another proof of the regularity of the (LB)-spaces $\cup_{n=1}^\infty \cO^m_{n,\omega}(\R^N)$.

\begin{prop}\label{P.LBregular} Let $\omega$ be a non-quasianalytic weight function and $m\in\N$. Then $\cup_{n=1}^\infty \cO^m_{n,\omega}(\R^N)$ is a regular  (LB)-space. Moreover, $\cup_{n=1}^\infty \cO^m_{n,\omega}(\R^N)$ is the bornological space associated with the space $(\cO^m_{s_\omega}(\R^N),\tau_m)$. 
	\end{prop}

\begin{proof} The result  follows from Proposition \ref{P.propOV}(2).  Indeed,   in the proof of Proposition  \ref{P.propOV}(2) it has been  established also that every bounded subset of $(\cO^m_{s_\omega}(\R^N),\tau_m)$ is contained and bounded in the Banch space $\cO^m_{n,\omega}(\R^N)$ for some $n\in\N$. 
	\end{proof}

Further immediate consequences of Proposition \ref{P.propOV} are the following results.

\begin{prop}\label{P.CompareT} Let $\omega$ be a non-quasianalytic weight function. Then the inclusion 
	\begin{equation}\label{eq.inclusioneT}
	(\cO_{M,\omega}(\R^N),t)\hookrightarrow (\cO_{M,\omega}(\R^N),\tau)
	\end{equation}
	is continuous, i.e., the topology $\tau$ is coarser than the topology $t$. Moreover, the spaces $(\cO_{M,\omega}(\R^N),t)$ and  $(\cO_{M,\omega}(\R^N),\tau)$ have the same bounded subsets.
\end{prop}

\begin{proof} Since 
\[
(\cO_{M,\omega}(\R^N),t)=\proj_{m \leftarrow }\cup_{n=1}^\infty\cO^m_{n,\omega}(\R^N) \ \ {\rm and }\ \ (\cO_{M,\omega}(\R^N),\tau)=\proj_{m\leftarrow}(\cO^m_{s_\omega}(\R^N),\tau_m)),
\]	
	the result follows immediately from Proposition \ref{P.propOV}(1)-(2).
	\end{proof}

\begin{prop}\label{P.incS} Let $\omega$ be a non-quasianalytic weight function. Then the inclusion 
	\begin{equation}\label{eq.inclusioneS}
	\cS_{\omega}(\R^N)\hookrightarrow (\cO_{M,\omega}(\R^N),\tau)
	\end{equation}
	is continuous with dense range.
\end{prop}
\begin{proof}
	The result immediately follows from Proposition $\ref{P.CompareT}$, Theorem $\ref{T.incl}$, and Theorem $\ref{den}$.
\end{proof}

Let $X$ be a Hausdorff lc-space and $\Gamma_X$ be a system of continuous seminorms generating the topology of $X$. Then the strong operator topology $\tau_s$ in the space $\cL(X)$ of all continuous linear operators from $X$ into itself is determined by the family of seminorms $q_x(S):=q(Sx)$ ($S\in \cL(X)$) for each $x\in X$ and $q\in \Gamma_X$. In such a case we write $\cL_s(X)$. Denoted by $\cB(X)$ the collection of all bounded subsets of $X$, the topology $\tau_b$ of uniform convergence on bounded sets is defined in $\cL(X)$ by the seminorms $q_B(S):=\sup_{x\in B}q(Sx)$ ($S\in \cL(X)$) for each $B\in \cB(X)$ and $q\in \Gamma_X$. In such a case we write $\cL_b(X)$.

By Theorem \ref{T.Multiplier} the space $\cO_{M,\omega}(\R^N)$ can be identified with the space $\cM(\cS_{\omega}(\R^N))$ of all multipliers on $\cS_\omega(\R^N)$ via the map $M\colon \cO_{M,\omega}(\R^N)\to \cM(\cS_{\omega}(\R^N))$ defined by $M(f):=M_f$ for each $f\in \cO_{M,\omega}(\R^N)$. Since  $\cM(\cS_{\omega}(\R^N))$ is a subspace of $\cL(\cS_{\omega}(\R^N))$, the space $\cO_{M,\omega}(\R^N)$ (via the map $M$) can be then endowed with either the topology $\tau_b$ induced by $\cL_b(\cS_{\omega}(\R^N))$ or the topology $\tau_s$ induced by $\cL_s(\cS_{\omega}(\R^N))$. In the next result we compare the three topologies $\tau_b$, $\tau_s$ and $\tau$.

To this end, 
we first show the following variant of Lemma \ref{se}.

\begin{lem}\label{L.Funzione} Let $\omega$ be a non-quasianalytic weight function. Let $h\colon \R^N\to \R$ be a non-negative function satisfying the condition
	\begin{equation}\label{eq.condition}
	\forall\lambda>0\quad 	\lim_{|x|\to\infty}\exp(\lambda \omega(x))h(x)=0.
	\end{equation}
	Then there exists $g\in\cS_\omega(\R^N)$ such that 
	\begin{equation}\label{eq.stimaNF}
	\forall x\in\R^N\quad 	h(x)\leq g(x).
	\end{equation}
\end{lem}

\begin{proof} Let $\rho\in \cD_\omega(\R^N)$  such that $\rho\geq 0$, $\rho\equiv 1$ on $B_1(0)$ and ${\rm supp}\, \rho\su B_1(0)$. Let $\{x_j\}_{j\in\N}\su \R^N$ be a sequence satisfying the following properties: $\lim_{j\to\infty}|x_j|=\infty$; there exists $H\in\N$ such that $|\{j\in\N\colon x\in B_2(x_j)\}|\leq H$ for every  $x\in\R^N$; for any $x\in\R^N$  there exists $j\in\N$ such that $x\in B_1(x_j)$.
	
	Let 
	\begin{equation}\label{eq.coeff}
	a_j:=\sup_{x\in \ov{B}_2(x_j)}h(x),\quad j\in\N,
	\end{equation}
	and
	\begin{equation}\label{eq.funz}
	g(x):=\sum_{j\in\N}a_j\rho(x-x_j), \quad x\in\R^N,
	\end{equation}
	where the series on the right of \eqref{eq.funz} is a finite sum for every $x\in\R^N$. 
	
	For any fixed $x\in\R^N$, let $j\in\N$ such that $x\in B_1(x_j)$. Then $g(x)\geq a_j\geq h(x)$. This shows that $h(x)\leq g(x)$ for all $x\in\R^N$. So, it remains to establish that $g\in \cS_\omega(\R^N)$. 
	
	Since the property $(\gamma)$ of $\omega$ implies that $\exp(\omega(x))\geq \exp(a)(1+|x|)^b$ for every $x\in\R^N$,   by  \cite[Lemma 3.6, p.127]{Ch} we can conclude that $g\in \cS(\R^N)$. In order to conclude the proof, we proceed as follows.
	
	Fix $\lambda, \mu>0$ and $x\in\R^N$. Then we have for every $\alpha\in\N_0^N$ that 
	\[
	\partial^\alpha g(x)=\sum_{x\in \ov{B}_2(x_j)}a_j\partial^\alpha\rho(x-x_j),
	\]
	where the set $J(x):=\{j\in\N\colon x\in \ov{B}_2(x_j)\}$ has cardinality less or equal to that of $H$, with $H$ indipendent of $x$.
	Since $\rho\in \cD_\omega(\R^N)$, it follows for every $\alpha\in\N_0^N$ that 
	\begin{align}\label{eq.nuovab}
	&\exp(\mu\omega(x))|\partial^\alpha g(x)| \leq \nonumber \\
	& \leq H\sup_{j\in J(x)}a_j\exp(\mu\omega(x))\exp\left(-\mu K\omega(x-x_j)+\lambda\varphi^*_\omega \left(\frac{|\alpha|}{\lambda}\right)\right)q_{\lambda, \mu K}(\rho),
	\end{align}
	where $K$ is the constant appearing in the property $(\alpha)$ of $\omega$. 
	On the other hand, by \eqref{eq.condition} there exists $C>0$ such that $h(x)\leq C \exp(-\mu K^2\omega(x))$ for all $x\in\R^N$ and so, we have for every $j\in\N$ that 
	\begin{equation}\label{eq.nuovac}
	a_j=\sup_{x\in \ov{B}_2(x_j)}h(x)=\sup_{y\in \ov{B}_2(0)}h(y+x_j)\leq C\sup_{y\in \ov{B}_2(0)}\exp(-\mu K^2\omega(y+x_j)).
	\end{equation}
	But, the following inequalities are satisfied
	\[
	\omega(x)=\omega((x-x_j)+x_j)\leq K(\omega(x-x_j)+\omega(x_j)+1)
	\]
	and
	\[
	\omega(x_j)=\omega((y+x_j)-y)\leq K(\omega(y+x_j)+\omega(y)+1).
	\]
	Accordingly, by \eqref{eq.nuovab} and \eqref{eq.nuovac} it follows for every $\alpha\in\N_0^N$ that 
	\begin{align*}
	&\exp(\mu\omega(x))|\partial^\alpha g(x)|
	\leq HC \sup_{j\in J(x)}\sup_{y\in \ov{B}_2(0)}\exp(K^2\mu(\omega(y)+1))\exp(-K\mu\omega(x_j))\times \\
	& \times e^{\mu K} \exp(\mu K\omega(x-x_j))\exp(\mu K\omega(x_j))\exp\left(-\mu K\omega(x-x_j)+\lambda\varphi^*_\omega \left(\frac{|\alpha|}{\lambda}\right)\right)q_{\lambda, \mu K}(\rho)\\
	&=H Ce^{\mu K}\sup_{y\in \ov{B}_2(0)}\exp(K^2\mu(\omega(y)+1))\exp\left(\lambda\varphi^*_\omega\left(\frac{|\alpha|}{\lambda}\right)\right)q_{\lambda,\mu K}(\rho)
	\end{align*}
	and hence,
	\[
	\exp\left(-\lambda\varphi^*_\omega\left(\frac{|\alpha|}{\lambda}\right)\right)\exp(\mu\omega(x))|\partial^\alpha g(x)|\leq H Ce^{\mu K}\sup_{y\in \ov{B}_2(0)}\exp(K^2\mu(\omega(y)+1))q_{\lambda,\mu K}(\rho).
	\]
	Since $D:=\sup_{y\in \ov{B}_2(0)}\exp(K^2\mu(\omega(y)+1))<\infty$ is a constant independent of $x$ and $x$ is arbitrary, we can conclude that  
	\[
	q_{\lambda,\mu}(g)\leq HCDe^{\mu K}q_{\lambda,\mu K}(\rho)<\infty.
	\]
	But, $\lambda$ and $\mu$ are also arbitrary. So, this implies that $g\in \cS_\omega(\R^N)$.
\end{proof}

\begin{thm}\label{T.Topol} Let $\omega$ be a non-quasianalytic weight function. Then the inclusions
\begin{equation}\label{eq.identita}
(\cO_{M,\omega}(\R^N),\tau)\hookrightarrow (\cO_{M,\omega}(\R^N),\tau_b)\hookrightarrow (\cO_{M,\omega}(\R^N),\tau_s)
\end{equation}
are continuous. Moreover, the spaces $(\cO_{M,\omega}(\R^N),\tau)$,  $(\cO_{M,\omega}(\R^N),\tau_b)$ and $(\cO_{M,\omega}(\R^N),\tau_s)$ have the same bounded subsets.
\end{thm}

\begin{proof} Since $\tau_s\su\tau_b$, it suffices to show that $\tau_b\su \tau$. To this end, let $W$ be a $0$-neighbourhood of $\cL_b(\cS_\omega(\R^N))$. Then there exist  a  $0$-neighbourhood $V$ of $\cS_\omega(\R^N)$ and a bounded subset $B$ of $\cS_\omega(\R^N)$ such that
	\[
	\{T\in \cL(\cS_\omega(\R^N))\colon T(B)\su V\}\su W.
	\]
	We can suppose that $V=\{h\in \cS_\omega(\R^N)\colon q_{m,n}(f)\leq\epsilon\}$ for some $m,n\in\N$ and $\epsilon>0$. To conclude the proof, we have to show that there exists a $0$-neighbourhood $U$ of $(\cO_{M,\omega}(\R^N),\tau)$ such that $fg\in V$ for all $f\in U$ and $g\in B$. To this end, let $M\in\N$ such that $M\geq Lm$ and $\epsilon'>0$ such that $\epsilon'<\epsilon e^{-mL}$, where $L\geq 1$ is the constant appearing in formula \eqref{eq.bb}. Then the set $U$ of all functions $f\in\cO_{M,\omega}(\R^N)$ satisfying the condition
	\begin{equation}\label{eq.intorno}
	  \sup_{x\in \R^N}\sup_{\gamma,\delta\in\N_0^N}\exp\left(n\omega(x)-M\varphi^*_\omega\left(\frac{|\gamma+\delta|}{M}\right)\right) |\partial^\gamma f(x)||\partial^\delta g(x)|\leq\epsilon', \quad g\in B,
	\end{equation}
is a $0$-neighbourhood of $(\cO_{M,\omega}(\R^N),\tau)$ for which $fg\in V$ for all $f\in U$ and $g\in B$. Indeed, for fixed  $f\in U$ and $g\in B$, we have  for every $\alpha\in\N_0^N$ and $x\in\R^N$ that
\begin{align*}
	|\partial^\alpha (fg)(x)|&\leq  \sum_{\gamma\leq \alpha }\binom{\alpha}{\gamma} |\partial^\gamma f(x)||\partial^{\alpha-\gamma} g(x)|\leq \epsilon' \sum_{\gamma\leq \alpha }\binom{\alpha}{\gamma}\exp\left(-n\omega(x)+M\varphi^*_\omega\left(\frac{|\alpha|}{M}\right)\right)\\
	&\leq \epsilon' 2^{|\alpha|}\exp(-n\omega(x))\exp\left(M\varphi^*_\omega\left(\frac{|\alpha|}{M}\right)\right)\leq \epsilon' e^{mL}\exp(-n\omega(x))\exp\left(m\varphi^*_\omega\left(\frac{|\alpha|}{m}\right)\right)\\
	&\leq \epsilon \exp\left(-n\omega(x)+m\varphi^*_\omega\left(\frac{|\alpha|}{m}\right)\right).
	\end{align*} 
	after having used inequality (\ref{firstprop}). Therefore, $q_{m,n}(fg)\leq \epsilon$ and so $fg\in V$. 

It remains to show that $U$ is a  $0$-neighbourhood of $(\cO_{M,\omega}(\R^N),\tau)$. To see this, we define
\[
h(x):=\exp(n\omega(x))\sup_{g\in B}\sup_{\delta\in \N_0^N}\exp\left(-M\varphi^*_\omega\left(\frac{|\delta|}{M}\right)\right)|\partial^\delta g(x)|,\quad x\in\R^N.
\] 
Since $B$ a bounded subset of $\cS_\omega(\R^N)$, we have for every $\lambda\geq 0$ and $x\in\R^N$ that 
\begin{align*}
\exp(\lambda \omega(x))h(x)&\leq \exp((\lambda+n)\omega(x))\sup_{g\in B}\sup_{\delta\in \N_0^N}\exp\left(-M\varphi^*_\omega\left(\frac{|\delta|}{M}\right)\right)|\partial^\delta g(x)|\\
&\leq \sup_{g\in B}q_{M,\lambda+n}(g)<\infty.
\end{align*}
Accordingly,  $h$ is a well-defined non-negative function on  $\R^N$  satisfying condition \eqref{eq.condition}. So, by Lemma \ref{L.Funzione} there exists $g_1\in \cS_\omega(\R^N)$ such that $h(x)\leq g_1(x)$ for all $x\in\R^N$. Then $U':=\{f\in \cO_{M,\omega}(\R^N)\colon q_{M,g_1}(f)\leq \epsilon'\}$ is a $0$-neighbourhood of $(\cO_{M,\omega}(\R^N),\tau)$. Moreover, for  fixed $f\in U'$ and $g\in B$, we have for every $\gamma,\delta\in\N_0^N$   that
\begin{align*}
|\partial^\gamma f(x)||\partial^\delta g(x)|&\leq |\partial^\gamma f(x)| g_1(x)\exp(-n\omega(x))\exp\left(M\varphi^*_\omega\left(\frac{|\delta|}{M}\right)\right)\\
&\leq \epsilon' \exp\left(M\varphi^*_\omega\left(\frac{|\gamma|}{M}\right)\right)\exp(-n\omega(x))\exp\left(M\varphi^*_\omega\left(\frac{|\delta|}{M}\right)\right)\\
&= \epsilon'\exp(-n\omega(x))\exp\left(M\varphi^*_\omega\left(\frac{|\gamma|}{M}\right)+M\varphi^*_\omega\left(\frac{|\delta|}{M}\right)\right)\\
&\leq \epsilon'\exp(-n\omega(x))\exp\left(M\varphi^*_\omega\left(\frac{|\gamma+\delta|}{M}\right)\right).
	\end{align*}
Since $f\in U'$ and $g\in B$ are arbitrary, this implies that  $U'\su U$. So, $U$ is a $0$-neighbourhood of $(\cO_{M,\omega}(\R^N),\tau)$.

	In order to show that the spaces $(\cO_{M,\omega}(\R^N),\tau)$, $(\cO_{M,\omega}(\R^N),\tau_b)$ and $(\cO_{M,\omega}(\R^N),\tau_s)$ have the same bounded sets, it suffices to prove only  that every bounded subset of  $(\cO_{M,\omega}(\R^N),\tau_s)$ is also a bounded subset of  $(\cO_{M,\omega}(\R^N),\tau)$. To this end, we  fix a bounded subset $B$ of $(\cO_{M,\omega}(\R^N),\tau_s)$ and suppose that $B$ is not a bounded subset of $(\cO_{M,\omega}(\R^N),\tau)$. Then there exist $g\in \cS_\omega(\R^N)$ and $m\in\N$ such that
	\begin{equation}\label{eq.boundedS}
	\sup_{f\in B}q_{m,g}(f)=\infty.
	\end{equation}
	 To get a contradiction, we first observe that the inclusion $(\cO_{M,\omega}(\R^N),\tau_s)\hookrightarrow \cE_\omega(\R^N)$ is also continuous. Indeed, fixed any compact subset $K$ of $\R^N$, let $h\in \cD_\omega(\R^N)$ such that $h\equiv 1$ on $K$. Then $\partial^\alpha (fh)=\partial^\alpha f$ on $K$ for each $f\in \cO_{M, \omega}(\R^N)$ and $\alpha\in \N_0^N$. This implies for every $f\in \cO_{M,\omega}(\R^N)$ and $l\in\N$ that 
	 \begin{align}\label{eq.conttt}
	 p_{K,m}(f)&=\sup_{x\in K}\sup_{\alpha\in\N_0^N}|\partial^\alpha f(x)|\exp\left(-l\varphi^*_\omega\left(\frac{|\alpha|}{l}\right)\right)\nonumber\\
	 &= \sup_{x\in K}\sup_{\alpha\in\N_0^N}|\partial^\alpha (fh)(x)|\exp\left(\omega(x)-l\varphi^*_\omega\left(\frac{|\alpha|}{l}\right)\right)\exp(\omega(x))\nonumber\\
	 &\leq C_Kq_{l,1}(fh),
	 \end{align}
	 where $C_K:=\sup_{x\in K}\exp(\omega(x))<\infty$ is a positive constant depending only on $K$.
Accordingly,   the inclusion $(\cO_{M,\omega}(\R^N),\tau_s)\hookrightarrow \cE_\omega(\R^N)$ is  continuous.

Let $K_j:=\ov{B}_j(0)\times\{\alpha\in\N_0^N\colon |\alpha|\leq j\}$ for each $j\in\N$ (so, $\cup_{j\in\N}K_j=\R^N\times \N_0^N$). The continuity of the inclusion $(\cO_{M,\omega}(\R^N),\tau_s)\hookrightarrow \cE_\omega(\R^N)$ implies for every $f\in B$ and $n,j\in\N$ that
\begin{align*}
&\sup_{(x,\alpha)\in K_j}|g(x)||\partial^\alpha f(x)|\exp\left(-m\varphi^*_\omega\left(\frac{|\alpha|}{m}\right)\right)\leq\\
&\leq k_j\sup_{(x,\alpha)\in K_j}|\partial^\alpha f(x)|\exp\left(-m\varphi^*_\omega\left(\frac{|\alpha|}{m}\right)\right)\leq k_j p_{\ov{B}_j(0),m}(f),
\end{align*}
where $k_j:=\sup_{x\in \ov{B}_j(0) }|g(x)|<\infty$, and hence
\[
\sup_{f\in B}\sup_{(x,\alpha)\in K_j}|g(x)||\partial^\alpha f(x)|\exp\left(-m\varphi^*_\omega\left(\frac{|\alpha|}{m}\right)\right)<\infty.
\]
Taking in mind  \eqref{eq.boundedS},  we can then  argument by induction to find a sequence $\{f_n\}_{n\in\N}\su B$ and  a strictly increasing sequence $\{j_n\}_{n\in\N}$ of positive integers such  that for every $n\in\N$ there exists some point $(x_n,\alpha_n)\in \stackrel{\circ}{K}_{j_{n+1}}\setminus K_{j_{n-1}}$ for which
\begin{equation}\label{eq.contrN}
|g(x_n)||\partial^{\alpha_n}f_n(x_n)|\exp\left(-m\varphi^*_\omega\left(\frac{|\alpha|}{m}\right)\right)>n.
\end{equation}
Now, we choose   a function $\rho\in\cD_\omega(\R^N)$ such that ${\rm supp}\rho\su B_1(0)$ and $\rho\equiv 1$ on $B_r(0)$ for some $0<r<1$.
We define 
\begin{equation}\label{eq.nuovafF}
g_0(x):=\sum_{k=1}^\infty g(x_k)\rho(x-x_k), \quad x\in \R^N. 
\end{equation}
Then by Remark \ref{R.se}  we have $g_0\in\cS_\omega(\R^N)$. In particular, $g_0\equiv g(x_n)$ in $\ov{B}_r(x_n)$ for each $n\in\N$. Accordingly, for every $n\in\N$, $x\in \ov{B}_r(x_n)$ and $\alpha\in\N_0^N$ we have
\[
\partial^{\alpha}(fg_0)(x)=\partial^{\alpha}f(x)g(x_n).
\]
Hence, by \eqref{eq.contrN} it follows for every $n\in\N$  that
\begin{align*}
q_{m,m}(f_ng_0)&\geq |g(x_n)||\partial^{\alpha_n} f_n(x_n)|\exp\left(m\omega(x)-m\varphi^*_\omega
\left(\frac{|\alpha_n|}{m}\right)\right)\\
&\geq |g(x_n)| |\partial^{\alpha_n} f_n(x_n)|\exp\left(-m\varphi^*_\omega
\left(\frac{|\alpha_n|}{m}\right)\right)>n.
\end{align*}
This shows that 
$	\sup_{f\in B}q_{m,m}(fg_0)=\infty$,
which is a contradiction because $B$ is a bounded subset of  $( \cO_{M,\omega}(\R^N),\tau_s)$. Hence $\sup_{f\in B}q_{m,g}(f)<\infty$. 	
\end{proof}

Finally, we have

\begin{prop}\label{P.nucleare} Let $\omega$ be a non-quasianalytic weight function. Then 
$(\cO_{M,\omega}(\R^N),\tau_b)$ and its strong dual are nuclear lc-spaces.	Moreover, $(\cO_{M,\omega}(\R^N),\tau_b)$ is complete.
\end{prop}

\begin{proof} Since $\cS_\omega(\R^N)$ is a nuclear Fr\'echet space by \cite[Theorem 3.3]{BJOS}, the space $\cL_b(\cS_\omega(\R^N))$ and its strong dual space are nuclear lc-spaces, \cite[Corollaire 3, Chap. II, \S 2, p.48]{G}. Therefore, by \cite[Th\'eor\`eme 9, Chap. II, \S 2, p.47]{G} $(\cO_{M,\omega}(\R^N),\tau_b)$ and its strong dual are also nuclear lc-spaces.
	
It is straightforward to show	the completeness of $(\cO_{M,\omega}(\R^N),\tau_b)$. 
\end{proof}

\begin{rem} It remains an open problem to establish which of the  topologies $t$, $\tau$, $\tau_b$ and $\tau_s$ coincide on the space $\cO_{M,\omega}(\R^N)$.
\end{rem}

\end{document}